\documentclass[12pt]{amsart}

\usepackage{graphics}
\usepackage{color}
\usepackage[a4paper, margin=3cm]{geometry}

\usepackage{amssymb,enumerate}
\usepackage{amsmath}
\usepackage{bbm}           
\usepackage{bm}
\usepackage{eso-pic,graphicx}
\usepackage{tikz}
\usepackage{cite}
\usepackage{esint}
\usepackage[colorlinks=true, pdfstartview=FitV, linkcolor=blue, citecolor=blue, urlcolor=blue]{hyperref}
\usepackage{booktabs}
\usepackage{graphicx}
\usepackage{pslatex}
\usepackage{amsmath,amsfonts}
\usepackage{rotating}
\usepackage{amssymb}
\usepackage{verbatim}
\usepackage{rotating}
\usepackage{mathrsfs}

\makeatletter
\def\Ddots{\mathinner{\mkern1mu\raise\p@
\vbox{\kern7\p@\hbox{.}}\mkern2mu
\raise4\p@\hbox{.}\mkern2mu\raise7\p@\hbox{.}\mkern1mu}}
\makeatother

\def\XXint#1#2#3{{\setbox0=\hbox{$#1{#2#3}{\int}$}
\vcenter{\hbox{$#2#3$}}\kern-.5\wd0}}

\begin{document}
\newtheorem{theorem}{Theorem}
\newtheorem{proposition}[theorem]{Proposition}
\newtheorem{conjecture}[theorem]{Conjecture}
\def\theconjecture{\unskip}
\newtheorem{corollary}[theorem]{Corollary}
\newtheorem{lemma}[theorem]{Lemma}
\newtheorem{claim}[theorem]{Claim}
\newtheorem{sublemma}[theorem]{Sublemma}
\newtheorem{observation}[theorem]{Observation}
\theoremstyle{definition}
\newtheorem{definition}{Definition}
\newtheorem{notation}[definition]{Notation}
\newtheorem{remark}[definition]{Remark}
\newtheorem{question}[definition]{Question}
\newtheorem{questions}[definition]{Questions}
\newtheorem{example}[definition]{Example}
\newtheorem{problem}[definition]{Problem}
\newtheorem{exercise}[definition]{Exercise}
 \newtheorem{thm}{Theorem}
 \newtheorem{cor}[thm]{Corollary}
 \newtheorem{lem}{Lemma}[section]
 \newtheorem{prop}[thm]{Proposition}
 \theoremstyle{definition}
 \newtheorem{dfn}[thm]{Definition}
 \theoremstyle{remark}
 \newtheorem{rem}{Remark}
 \newtheorem{ex}{Example}
 \numberwithin{equation}{section}
\def\C{\mathbb{C}}
\def\R{\mathbb{R}}
\def\Rn{{\mathbb{R}^n}}
\def\Rns{{\mathbb{R}^{n+1}}}
\def\Sn{{{S}^{n-1}}}
\def\M{\mathbb{M}}
\def\N{\mathbb{N}}
\def\Q{{\mathbb{Q}}}
\def\Z{\mathbb{Z}}
\def\F{\mathcal{F}}
\def\L{\mathcal{L}}
\def\S{\mathcal{S}}
\def\supp{\operatorname{supp}}
\def\essi{\operatornamewithlimits{ess\,inf}}
\def\esss{\operatornamewithlimits{ess\,sup}}

\numberwithin{equation}{section}
\numberwithin{thm}{section}
\numberwithin{theorem}{section}
\numberwithin{definition}{section}
\numberwithin{equation}{section}

\def\earrow{{\mathbf e}}
\def\rarrow{{\mathbf r}}
\def\uarrow{{\mathbf u}}
\def\varrow{{\mathbf V}}
\def\tpar{T_{\rm par}}
\def\apar{A_{\rm par}}

\def\reals{{\mathbb R}}
\def\torus{{\mathbb T}}
\def\scriptm{{\mathcal T}}
\def\heis{{\mathbb H}}
\def\integers{{\mathbb Z}}
\def\z{{\mathbb Z}}
\def\naturals{{\mathbb N}}
\def\complex{{\mathbb C}\/}
\def\distance{\operatorname{distance}\,}
\def\support{\operatorname{support}\,}
\def\dist{\operatorname{dist}\,}
\def\Span{\operatorname{span}\,}
\def\degree{\operatorname{degree}\,}
\def\kernel{\operatorname{kernel}\,}
\def\dim{\operatorname{dim}\,}
\def\codim{\operatorname{codim}}
\def\trace{\operatorname{trace\,}}
\def\Span{\operatorname{span}\,}
\def\dimension{\operatorname{dimension}\,}
\def\codimension{\operatorname{codimension}\,}
\def\nullspace{\scriptk}
\def\kernel{\operatorname{Ker}}
\def\ZZ{ {\mathbb Z} }
\def\p{\partial}
\def\rp{{ ^{-1} }}
\def\Re{\operatorname{Re\,} }
\def\Im{\operatorname{Im\,} }
\def\ov{\overline}
\def\eps{\varepsilon}
\def\lt{L^2}
\def\diver{\operatorname{div}}
\def\curl{\operatorname{curl}}
\def\etta{\eta}
\newcommand{\norm}[1]{ \|  #1 \|}
\def\expect{\mathbb E}
\def\bull{$\bullet$\ }

\def\blue{\color{blue}}
\def\red{\color{red}}

\def\xone{x_1}
\def\xtwo{x_2}
\def\xq{x_2+x_1^2}
\newcommand{\abr}[1]{ \langle  #1 \rangle}

\newcommand{\Norm}[1]{ \left\|  #1 \right\| }
\newcommand{\set}[1]{ \left\{ #1 \right\} }
\newcommand{\ifou}{\raisebox{-1ex}{$\check{}$}}
\def\one{\mathbf 1}
\def\whole{\mathbf V}
\newcommand{\modulo}[2]{[#1]_{#2}}
\def \essinf{\mathop{\rm essinf}}
\def\scriptf{{\mathcal F}}
\def\scriptg{{\mathcal G}}
\def\scriptm{{\mathcal M}}
\def\scriptb{{\mathcal B}}
\def\scriptc{{\mathcal C}}
\def\scriptt{{\mathcal T}}
\def\scripti{{\mathcal I}}
\def\scripte{{\mathcal E}}
\def\scriptv{{\mathcal V}}
\def\scriptw{{\mathcal W}}
\def\scriptu{{\mathcal U}}
\def\scriptS{{\mathcal S}}
\def\scripta{{\mathcal A}}
\def\scriptr{{\mathcal R}}
\def\scripto{{\mathcal O}}
\def\scripth{{\mathcal H}}
\def\scriptd{{\mathcal D}}
\def\scriptl{{\mathcal L}}
\def\scriptn{{\mathcal N}}
\def\scriptp{{\mathcal P}}
\def\scriptk{{\mathcal K}}
\def\frakv{{\mathfrak V}}
\def\C{\mathbb{C}}
\def\D{\mathcal{D}}
\def\R{\mathbb{R}}
\def\Rn{{\mathbb{R}^n}}
\def\rn{{\mathbb{R}^n}}
\def\Rm{{\mathbb{R}^{2n}}}
\def\r2n{{\mathbb{R}^{2n}}}
\def\Sn{{{S}^{n-1}}}
\def\M{\mathbb{M}}
\def\N{\mathbb{N}}
\def\Q{{\mathcal{Q}}}
\def\Z{\mathbb{Z}}
\def\F{\mathcal{F}}
\def\L{\mathcal{L}}
\def\G{\mathscr{G}}
\def\ch{\operatorname{ch}}
\def\supp{\operatorname{supp}}
\def\dist{\operatorname{dist}}
\def\essi{\operatornamewithlimits{ess\,inf}}
\def\esss{\operatornamewithlimits{ess\,sup}}
\def\dis{\displaystyle}
\def\dsum{\displaystyle\sum}
\def\dint{\displaystyle\int}
\def\dfrac{\displaystyle\frac}
\def\dsup{\displaystyle\sup}
\def\dlim{\displaystyle\lim}
\def\bom{\Omega}
\def\om{\omega}

\author[F. Liu]{Feng Liu}
\address{Feng Liu:
College of Mathematics and System Science\\
Shandong University of Science and Technology\\
Qingdao, Shandong 266590\\
People's Republic of China}
\email{FLiu@sdust.edu.cn}

\author[S. Liu]{Simin Liu}
\address{Simin Liu:
College of Mathematics and System Science\\
Shandong University of Science and Technology\\
Qingdao, Shandong 266590\\
People's Republic of China}
\email{13734187503@163.com}
	
\author[M. Shi]{Minglei Shi$^{*}$}
\address{Minglei Shi:
        Department of Mathematics\\
        Zhejiang Shuren University\\
        Shaoxing, Zhejiang 312028\\
        People's Republic of China}
\email{stfoursml@gmail.com}

\keywords{Weak type $(1,\,1)$ bound criterion, singular integral, rough kernel,
Christ--Journ\'{e} type commutator.\\
\indent{2020 Mathematics Subject Classification.} Primary 42B20, 42B25.}
	
\thanks{$^{*}$ Corresponding author, e-mail address: stfoursml@gmail.com}

\date{\today}
\title[A criterion on weak type $(1,\,1)$ bound of rough singular integrals]
{A criterion on weak type $(1,\,1)$ bound of rough singular integrals}
\maketitle
	
\begin{abstract}
In this paper we establish a criterion on weak type $(1,\,1)$ bound of the
following rough singular integral
$$T_{\Omega,K,a}f(x)={\rm p.v.}\int_{\mathbb{R}^n}\Omega(x-y)K(x,y)m_{x,y}a f(y)dy,$$
where $m_{x,y}a=\int_0^1a(sx+(1-s)y)ds$ with $a\in L^1(\mathbb{R}^n)$ and
$\hat{a}\in L^1(\mathbb{R}^n)$, $\Omega$ is homogeneous of degree zero,
integrable in $\mathbb{S}^{n-1}$ and satisfies the cancellation condition
$\int_{\mathbb{S}^{n-1}}\Omega(\theta)d\sigma(\theta)=0$ and $K$ is a
measurable function defined on $\mathbb{R}^n\times\mathbb{R}^n\setminus
\{(x,x):x\in\mathbb{R}^n\}$ and satisfies a H\"{o}lder condition. By assuming 
that $\Omega\in L\log L(\mathbb{S}^{n-1})$ and the operator $T_{\Omega,
K}f(x)={\rm p.v.}\int_{\mathbb{R}^n}\Omega(x-y)K(x,y)f(y)dy$ is bounded on
$L^2(\mathbb{R}^n)$, we prove the weak type (1,1) bound of $T_{\Omega,K,a}$. 
As several applications, we obtain a large class of singular integral operators 
which possess weak type $(1,\,1)$ bound. The main results of this paper 
essentially extend and generalize some known ones.
\end{abstract}

\bigskip
	
\section{Introduction}\label{S1}
	
\medskip

Singular integral theory is one of the core components in harmonic analysis,
and plays a key role in the study of partial differential equations, complex
analysis, and other fields. Over the past several years considerable attention has been devoted to establishing the mapping properties for singular integrals
and related operators on Lebesgue spaces. 
It is well known that many important integral operators are not bounded on $L^1$, including the Hilbert transform, Riesz transforms, the Hardy–Littlewood maximal operator, maximal operators with rough kernels, and standard Calder\'{o}n--Zygmund singular integral operators and so on.
It makes the investigation on the $L^1$ behaviors
for various integral operators more complex and interesting. Recently, a
considerable amount of attention has been given to investigating the weak type
$(1,\,1)$ bound for various singular integral operators (see \cite{CG,CLW,CWW,Ch,
CR,CS,DL,Hof2,Hon,HLTX,Se,See,Tao}). In this paper we continue to study this
topic. More precisely, we shall establish a criterion on weak type $(1,\,1)$
bound of a large class of singular integral operators with rough kernels
in $L\log L(\mathbb{S}^{n-1})$.

Let $\Omega$ be homogeneous of degree zero, integrable in the unit sphere
$\mathbb{S}^{n-1}$ and satisfy the cancellation condition
$$\int_{\mathbb{S}^{n-1}}\Omega(\theta)d\sigma(\theta)=0.\eqno(1.1)$$
The classical singular integral operator $T_\Omega$ is defined by
$$T_{\Omega}f(x)={\rm p.v.}\int_{\mathbb{R}^n}\frac{\Omega(x-y)}{|x-y|^n}f(y)dy.$$
As is well known, $T_{\Omega}$ is a typical convolution singular integral
operator, which was introduced by Calder\'{o}n and Zygmund \cite{CZ} who
established the $L^p$ boundedness for $T_\Omega$ for $1<p<\infty$ by
assuming $\Omega\in L\log L(\mathbb{S}^{n-1})$, that is
$$\int_{\mathbb{S}^{n-1}}|\Omega(\theta)|\log({\rm e}+|\Omega(\theta)|)d\sigma(\theta)<\infty.$$
Later on, the above result was independently improved by Ricci and Weiss
\cite{RW} and Connett \cite{Con} to $\Omega\in H^1(\mathbb{S}^{n-1})$ (the
Hardy space on unit sphere). For the case $p=1$, when $n=2$, Christ \cite{Ch}
and Hofmann \cite{Ho1} independently obtained the weak type $(1,\,1)$ bound for
$T_\Omega$ under the condition $\Omega\in L^q(\mathbb{S}^{n-1})$ for some $q>1$.
Subsequently, the above result was improved by Seeger \cite{Se} to the case
$\Omega\in L\log L(\mathbb{S}^{n-1})$ and $n\geq2$. Very recently, Chen et
al. \cite{CWW} established the weak type $(1,\,1)$ boundedness for $T_\Omega$
under the condition $\Omega\in B_q^{0,0}(\mathbb{S}^{n-1})$ (the block spaces).
Other extensions can be found in \cite{CZ4,DL,Ste,Tao}.

The following proper inclusion relations are well known:
$$L^\infty(\mathbb{S}^{n-1})\subsetneq L^r(\mathbb{S}^{n-1})\,(1<r<\infty)\subsetneq L\log L(\mathbb{S}^{n-1})\subsetneq H^1(\mathbb{S}^{n-1})\subsetneq L^1(\mathbb{S}^{n-1}),$$
$$L^q(\mathbb{S}^{n-1})\subsetneq B_q^{0,0}(\mathbb{S}^{n-1})\subsetneq H^1(\mathbb{S}^{n-1}),\ \ 1<q<\infty,$$
$$B_{q_1}^{0,0}(\mathbb{S}^{n-1})\subsetneq B_{q_2}^{0,0}(\mathbb{S}^{n-1}),\ \ \forall 1<q_2<q_1<\infty,$$
$$B_q^{0,0}(\mathbb{S}^{n-1})\nsubseteq L(\log L)^{1+\epsilon}(\mathbb{S}^{n-1}),\ \ \epsilon>0.$$
Very recently, Chen et al. \cite{CJWW} proved that
$$B_q^{0,0}(\mathbb{S}^{n-1})\subsetneq L\log L(\mathbb{S}^{n-1}),$$
which is a proper inclusion.

In contrast to the extensive literature on rough singular integrals, comparatively little is known about Christ-–Journ\'{e} type commutators with rough kernels.
Recall that the usual Christ--Journ\'{e} type commutator
with rough kernels is defined by
$$T_{\Omega,a}f(x)={\rm p.v.}\int_{\mathbb{R}^n}\frac{\Omega(x-y)}{|x-y|^n}m_{x,y}af(y)dy,$$
where $m_{x,y}a=\int_0^1a(sx+(1-s)y)ds$ with $a\in L^\infty(\mathbb{R}^n)$.
This type of commutator $T_{\Omega,a}$ was introduced by Hofmann, \cite{Hof1} who established its weighted $L^p$ boundedness under the assumption that $\Omega\in L^\infty(\mathbb{S}^{n-1})$. To the best of our knowledge, it remains unknown
whether $T_{\Omega,a}$ is of weak type $(1,\,1)$ for rough
kernels $\Omega$, even when $\Omega\in L^\infty(\mathbb{S}^{n-1})$
and $a\in L^\infty(\mathbb{R}^n)$. However, Ding and Lai
\cite{DL3} established the weak type $(1,\,1)$ bound and $L^p$ boundedness
of $T_{\Omega,a}$ under the condition $\Omega\in L\log L(\mathbb{S}^{n-1})$
by strengthening the symbol condition $a$. The main result of \cite{DL3} can
be listed as follows.

\medskip

\quad\hspace{-20pt}{\bf Theorem A} (\cite{DL3}) {\it Let $\Omega\in
L\log L(\mathbb{S}^{n-1})$ and satisfy $(1.1)$. Assume that
$a\in L^1(\mathbb{R}^n)$ and $\hat{a}\in L^1(\mathbb{R}^n)$, then
$T_{\Omega,a}$ is bounded on $L^p(\mathbb{R}^n)$ for $1<p<\infty$
and is of weak type $(1,\,1)$.}

\medskip

It is notable that the condition $\{a,\hat{a}\}\subset L^1(\mathbb{R}^n)$
implies $a\in L^\infty(\mathbb{R}^n)$. However, the converse is false.

The primary purpose of this paper is to present some new results focusing on
the weak type $(1,\,1)$ bound of the following rough singular integral
$$T_{\Omega,K,a}f(x)={\rm p.v.}\int_{\mathbb{R}^n}\Omega(x-y)K(x,y)m_{x,y}a f(y)dy,$$
where $m_{x,y}a=\int_0^1a(sx+(1-s)y)ds$ with $\{a,\hat{a}\}\subset L^1(\mathbb{R}^n)$
and $K$ satisfies a H\"{o}lder condition. We say that $K(x,y)$ satisfies a condition H\"{o}lder
 if there exist $\delta\in(0,1]$ and $C>0$ such that
$$|K(x,y)|\leq\frac{C}{|x-y|^n},\eqno(1.2)$$
$$|K(x_1,y)-K(x_2,y)|\leq C\frac{|x_1-x_2|^\delta}{|x_1-y|^{n+\delta}},\ \ \ {\rm if}\ |x_1-y|>2|x_1-x_2|,\eqno(1.3)$$
$$|K(x,y_1)-K(x,y_2)|\leq C\frac{|y_1-y_2|^\delta}{|x-y_1|^{n+\delta}},\ \ \ {\rm if}\ |x-y_1|>2|y_1-y_2|.\eqno(1.4)$$

It should be pointed out that the above singular integral operator $T_{\Omega,K,a}$
contains some classical examples. For example, when $K(x,y)=\frac{1}{|x-y|^{n}}$,
the operator $T_{\Omega,K,a}$ is just the Christ--Journ\'{e} type commutator with
rough kernels $T_{\Omega,a}$. Clearly, the function $K(x,y)=\frac{1}{|x-y|^{n}}$
satisfies the conditions (1.2)--(1.4). Moreover, from Theorem A we see that
$T_{\Omega,K,a}$ is bounded on $L^p(\mathbb{R}^n)$ for $1<p<\infty$ and of weak
type $(1,\,1)$ by assuming that $\Omega\in L\log L(\mathbb{S}^{n-1})$ satisfies
$(1.1)$ and $\{a,\hat{a}\}\subset L^1(\mathbb{R}^n)$ in the special case
$K(x,y)=\frac{1}{|x-y|^{n}}$.

It is natural and interesting to ask the following question.

\begin{question}\label{que1.1}
Let $\Omega\in L\log L(\mathbb{S}^{n-1})$ satisfy $(1.1)$ and
$\{a,\hat{a}\}\subset L^1(\mathbb{R}^n)$. Is $T_{\Omega,K,a}$ bounded on
$L^p(\mathbb{R}^n)$ for $1<p<\infty$ and of weak type $(1,\,1)$ if $K$
satisfies a H\"{o}lder condition?
\end{question}

This is the main motivation of this paper. In this paper, we will give an
affirmative answer to this question by assuming an a priori $L^2$ estimate
for the following operator
$$T_{\Omega,K}f(x)={\rm p.v.}\int_{\mathbb{R}^n}\Omega(x-y)K(x,y)f(y)dy,\ \ \ x\in\mathbb{R}^n.$$

\begin{remark}\label{rem1.1}
The operators $T_{\Omega,K}$ and $T_{\Omega,K,a}$ are closely related. Indeed, by the Fourier inversion formula, we have
$$m_{x,y}a=\frac{1}{(2\pi)^n}\iint_{[0,1]\times\mathbb{R}^n}\hat{a}(\eta)e^{is\langle\eta,x\rangle}e^{i(1-s)\langle y,\eta\rangle}d\eta ds.$$
This together with Fubini's theorem implies
$$T_{\Omega,K,a}f(x)=\iint_{[0,1]\times\mathbb{R}^n}a^{x,t}(\eta)T_{\Omega,K}(W^{\eta,t}f)(x)d\eta dt,\eqno(1.5)$$
where
$$a^{x,t}(\eta)=\frac{1}{(2\pi)^n}\hat{a}(\eta)e^{it\langle x,\eta\rangle},\ \ \ W^{\eta,t}(y)=e^{i(1-t)\langle y,\eta\rangle}.$$
By (1.5) and Minkowski's inequality, one obtains
$$\|T_{\Omega,K,a}\|_{L^p(\mathbb{R}^n)\rightarrow L^p(\mathbb{R}^n)}
\leq\|\hat{a}\|_1\|T_{\Omega,K}\|_{L^p(\mathbb{R}^n)\rightarrow L^p(\mathbb{R}^n)},\ \ 1<p<\infty.\eqno(1.6)$$
However, the weak type $(1,\,1)$ bound for $T_{\Omega,K}$ cannot imply
the weak type $(1,\,1)$ bound for $T_{\Omega,K,a}$.
\end{remark}

Now we establish a criterion for the weak type $(1,\,1)$ bound of $T_{\Omega,K,a}$.
\begin{theorem}\label{thm1.1}
Let $\Omega\in L\log L(\mathbb{S}^{n-1})$ satisfy $(1.1)$ and $K$ satisfy
H\"{o}lder's condition $(1.2)-(1.4)$. Assuming that $\{a,\hat{a}\}\subset
L^1(\mathbb{R}^n)$ and $T_{\Omega,K}$ is a bounded operator on
$L^2(\mathbb{R}^n)$ with bound $C\|\Omega\|_{L\log L(\mathbb{S}^{n-1})}$,
then we have
$$\|T_{\Omega,K,a}\|_{L^1(\mathbb{R}^n)\rightarrow L^{1,\infty}(\mathbb{R}^n)}\lesssim C_\Omega \|\hat{a}\|_1,$$
where
$$C_\Omega=\|\Omega\|_{L\log L(\mathbb{S}^{n-1})}+\int_{\mathbb{S}^{n-1}}|\Omega(\theta)|\Big(1+\log^{+}\frac{|\Omega(\theta)|}{\|\Omega\|_{L^1(\mathbb{S}^{n-1})}}\Big)d\sigma(\theta).$$
\end{theorem}

As direct consequences of Theorem \ref{thm1.1}, we have Theorem A and
the following conclusions.

\begin{corollary}\label{cor1.1}
Let $\Omega\in L\log L(\mathbb{S}^{n-1})$ satisfy $(1.1)$. Assume $\{a,
\hat{a}\}\subset L^1(\mathbb{R}^n)$, then $T_{\Omega,K,a}$ is bounded on
$L^p(\mathbb{R}^n)$ for $1<p<\infty$ and of weak type $(1,\,1)$, provided
that one of the following conditions holds:
\begin{enumerate}[{\rm (i)}]
\item $K(x,y)=\frac{A(x)-A(y)}{|x-y|^{n+1}}$ with $A\in Lip(\mathbb{R}^n)$,
$\Omega$ has vanishing moments of order $1$, i.e. for all multi-indices
$\gamma\in\mathbb{Z}_{+}^n$,
$$\int_{\mathbb{S}^{n-1}}\Omega(\theta)\theta^\gamma d\sigma(\theta)=0,\ \ \ \ \forall |\gamma|=1;\eqno(1.7)$$
\item $K(x,y)=\frac{1}{|x-y|^n}F\big(\frac{A(x)-A(y)}{|x-y|}\big)$, where
$A\in Lip(\mathbb{R}^n)$ and $F(t)=F(-t)$ for $t\in\mathbb{R}$ and
$F(t)$ is real analytic in $\{|t|\leq\|\nabla A\|_\infty\}$, $\Omega$ is
an odd function defined on $\mathbb{S}^{n-1}$;
\item $K(x,y)=\frac{1}{|x-y|^{n+ir}}$ with $r\in\mathbb{R}\setminus\{0\}$,
$\Omega$ satisfies $(1.1)$.
\end{enumerate}
\end{corollary}

\begin{remark}\label{rem1.2}
\begin{enumerate}[{\rm (i)}]
\item Theorem \ref{thm1.1} implies Theorem A, which corresponds to the
case $K(x,y)=\frac{1}{|x-y|^n}$. Actually, when $K(x,y)=\frac{1}{|x-y|^n}$,
the operator $T_{\Omega,K}$ is just the classical rough singular integral
operator $T_\Omega$. A celebrated result of Calder\'{o}n and Zygmund
\cite{CZ} states that $T_{\Omega}$ is bounded on $L^p(\mathbb{R}^n)$ for
$1<p<\infty$ if $\Omega\in L\log L(\mathbb{S}^{n-1})$ satisfies (1.1).
These facts together with Theorem \ref{thm1.1} and Remark \ref{rem1.1}
yield Theorem A.
\item Corollary \ref{cor1.1} follows from Theorem \ref{thm1.1}, Remark
\ref{rem1.1} and the following facts:
\begin{enumerate}[{\rm (a)}]
\item Let $K(x,y)=\frac{A(x)-A(y)}{|x-y|^{n+1}}$, where $\nabla A\in L^\infty
(\mathbb{R}^n)$. Pan et al. \cite{PWY} proved that $T_{\Omega,K}$ is
bounded on $L^p(\mathbb{R}^n)$ for $1<p<\infty$ if
$\Omega\in H^1(\mathbb{S}^{n-1})$. This together with the fact that
$L\log L(\mathbb{S}^{n-1})\subsetneq H^1(\mathbb{S}^{n-1})$ implies
the $L^p\,(1<p<\infty)$ boundedness for $T_{\Omega,A}$ under the condition
that $\Omega\in L\log L(\mathbb{S}^{n-1})$.
\item Let $K(x,y)=\frac{1}{|x-y|^n}F\big(\frac{A(x)-A(y)}{|x-y|}\big)$,
where $F,\,A$ are as in Corollary \ref{cor1.1}. In \cite{CCFJR},
Calder\'{o}n et al. concluded that $T_{\Omega,K}$ is bounded on
$L^p(\mathbb{R}^n)$ for $1<p<\infty$ if $\Omega\in L^1(\mathbb{S}^{n-1})$
is an odd function. It follows that $T_{\Omega,K}$ is bounded on
$L^p(\mathbb{R}^n)$ for $1<p<\infty$ when $\Omega\in L\log L(\mathbb{S}^{n-1})$.
\item Let $K(x,y)=\frac{1}{|x-y|^{n+ir}}$, where $r\in\mathbb{R}\setminus\{0\}$.
In \cite{Mu}, Muckenhoupt proved that $T_{\Omega,ir}$ is bounded on
$L^p(\mathbb{R}^n)$ for $1<p<\infty$ if $\Omega\in L^1(\mathbb{S}^{n-1})$
satisfies (1.1). It follows that $T_{\Omega,ir}$ is bounded on
$L^p(\mathbb{R}^n)$ for $1<p<\infty$ when $\Omega\in L\log L(\mathbb{S}^{n-1})$.
\end{enumerate}
\item By the fact that $B_q^{0,0}(\mathbb{S}^{n-1})\subsetneq L\log L(\mathbb{S}^{n-1})$,
our main results also hold for the case $\Omega\in\bigcup_{1<q<\infty}B_q^{0,0}(\mathbb{S}^{n-1})$.
\end{enumerate}
\end{remark}

This paper is organized as follows. In Section \ref{S2} we present some
notation and establish some technical lemmas, which are the main
ingredients for the proof of Theorem \ref{thm1.1}. The proof of Theorem
\ref{thm1.1} will be given in Section \ref{S3}. It should be pointed out
that some ideas of our methods are taken from \cite{DL3}, \cite{DL}, but
our methods and techniques are more delicate and complex than those of
\cite{DL3},\cite{DL}.

Throughout this paper, the notation $X\lesssim Y$ means $X\le C Y$ for some
constant $C>0$ that is independent of the essential variables involved in $X$ and $Y$; and $X\simeq Y$ means $X\lesssim Y\lesssim X$. We shall
use the following notation. For each $A\subset\mathbb{R}^n$ we denote
$A^c=\mathbb{R}^n\setminus A$. For $a\in\mathbb{R}$ we denote by $[a]$
the integer part of $a$. For $x=(x_1,\ldots,x_n)\in\mathbb{R}^n$ and
$y=(y_1,\ldots,y_n)\in\mathbb{R}^n$, we denote
$\langle x,y\rangle=\sum_{i=1}^nx_iy_i$. Given $f\in L^p(\mathbb{R}^n)$
for $1\leq p\leq\infty$, we denote $\|f\|_p=\|f\|_{L^p(\mathbb{R}^n)}$.
For $\alpha=(\alpha_1,\ldots,\alpha_n)\in\mathbb{Z}_{+}^n$ and
$x=(x_1,\dots,x_n)\in\mathbb{R}^n$, we write
$$|\alpha|=\alpha_1+\cdots+\alpha_n,\ \ \bigtriangleup_{x}=\partial_{x_1}^2+\cdots+\partial_{x_n}^2,\ \ \bigtriangledown_x=(\partial_{x_1}^1,\ldots,\partial_{x_n}^1),$$
$$\partial^\alpha f(x)=\partial_{x_1}^{\alpha_1}\cdots\partial_{x_n}^{\alpha_n}f(x).$$
We denote $\tilde{f}(x)=f(-x)$ and $\hat{f}$ (resp., $\mathcal{F}^{-1}(f)$)
the (resp., inverse) Fourier transform of $f$, i.e.
$$\mathcal{F}(f)(\xi)=\hat{f}(\xi)=\int_{\mathbb{R}^n}e^{-i\langle x,\xi\rangle}f(x)dx,\ \ \ \ \mathcal{F}^{-1}(f)(x)=\int_{\mathbb{R}^n}e^{i\langle x,
\xi\rangle}f(x)dx.$$
Given two functions $f,\,g$ defined on $\mathbb{R}^n$, the convolution
of $f$ and $g$ is given by
$$f*g(x)=\int_{\mathbb{R}^n}f(x-y)g(y)dy,\ \ \ x\in\mathbb{R}^n.$$
For a dyadic cube $Q$ we use the notation $L(Q)$ to denote the side length
of $Q$.

\bigskip

\section{Preliminary notation and lemmas}\label{S2}

\medskip

\subsection{Calder\'{o}n--Zygmund decomposition of $f W^{\eta,t}$}
We use the notation $\mathcal{M}_{HL}$ to denote the Hardy--Littlewood
maximal operator. The following Calder\'{o}n--Zygmund decomposition was
proved in \cite{DL3}.

\begin{lemma}\label{lem2.1}
{\rm (Lemma 2.1,\cite{DL3})} Let $(\eta,t)\in\mathbb{R}^n\times[0,1]$
and $W^{\eta,t}$ be as in $(1.5)$. Let $f\in L^1(\mathbb{R}^n)$. Set
$$\Omega_\tau=\{x\in\mathbb{R}^n:\mathcal{M}_{HL}f(x)>\tau\}.$$
There exist constants $C>0$ independent of $\eta,\,t$ such that
\begin{enumerate}[{\rm (i)}]
\item $\Omega_\tau=\bigcup_{Q\in\mathcal{B}}Q$, where the
$Q$'s are disjoint dyadic cubes and $\mathcal{B}$ is an 
index set;
\item $|\Omega_\tau|\leq C\tau^{-1}\|f\|_1$;
\item $f W^{\eta,t}=g^{\eta,t}+b^{\eta,t}$, where $\|g^{\eta,
t}\|_2^2\leq C\tau\|f\|_1$ and $\|b^{\eta,t}\|_1\leq C\|f\|_1$;
\item $b^{\eta,t}=\sum_{Q\in\mathcal{B}}b_Q^{\eta,t}$,
${\rm supp}(b_Q^{\eta,t})\subset Q,\ \int_{Q}b_Q^{\eta,t}(y)dy=0$,\ \
$\|b_Q^{\eta,t}\|_1\leq C\tau|Q|$.
\end{enumerate}
\end{lemma}

In what follows, let $f\in L^1(\mathbb{R}^n),\,b^{\eta,t}$
be as in Lemma \ref{lem2.1}. Set
$$B_j^{\eta,t}=\sum_{L(Q)=2^j} b_Q^{\eta,t},\ \ \ \ j\in\mathbb{Z}.\eqno(2.1)$$
Then we have $b^{\eta,t}=\sum_{j\in\mathbb{Z}}B_j^{\eta,t}$. Invoking
Lemma \ref{lem2.1}, we see that
$$\sum\limits_{j\in\mathbb{Z}}\sum\limits_{L(Q)=2^{j-s}}\|b_Q^{\eta,t}\|_1=\sum_{j\in\mathbb{Z}}\|B_j^{\eta,t}\|_1\leq\|b^{\eta,t}\|_1\leq\|f\|_1.\eqno(2.2)$$

\subsection{Some technical lemmas}Throughout this subsection, we assume
$a\in L^1(\mathbb{R}^n)$ and $\hat{a}\in L^1(\mathbb{R}^n)$. Let
$f,\,b_Q^{\eta,t}$ be as in Lemma \ref{lem2.1} and $B_{j}^{\eta,t}$
be defined in (2.1). Let $\beta$ be a nonnegative radial
$C_c^\infty(\mathbb{R}^n)$ function such that ${\rm supp}(\beta)\subset\{x\in
\mathbb{R}^n:1/2\leq|x|\leq2\}$ and $\sum_{j\in\mathbb{Z}}\beta_j(x)=1$
for all $x\in\mathbb{R}^n\setminus\{0\}$, where $\beta_j(x)=\beta(2^{-j}x)$.
Let $K$ be as in Theorem \ref{thm1.1} and set
$$K_j(x,y)=\beta_j(x-y)K(x,y).\eqno(2.3)$$
Let $\vartheta$ be a nonnegative radial $C_c^\infty(\mathbb{R}^n)$
function and satisfy ${\rm supp}(\vartheta)\subset\{x\in\mathbb{R}^n:
|x|\leq1\}$ and $\int_{\mathbb{R}^n}\vartheta(x)dx=1$. Let $s\geq1$
and $l_\delta(s)=[2\delta^{-1}\log_2 s]+2$ with $\delta\in(0,1)$. Set
$\vartheta_i(x)=2^{-in}\vartheta(2^{-i}x)$. Define a kernel $K_j^{s,
\delta}$ by
$$K_j^{s,\delta}(x,y)=\int_{\mathbb{R}^n}\vartheta_{j-l_\delta(s)}(x-z)K_j(z,y)dz.\eqno(2.4)$$

For $(\eta,t)\in\mathbb{R}^n\times[0,1]$ and $u^{\eta,t}:\mathbb{R}^n
\rightarrow\mathbb{R}$ we define the following operator
$$\Theta_{\Omega,K,a}u^{\eta,t}(x)=\iint_{[0,1]\times\mathbb{R}^n}a^{x,t}(\eta)T_{\Omega,K}u^{\eta,t}(x)d\eta dt.\eqno(2.5)$$

\begin{lemma}\label{lem2.2}
Let $\Omega\in L^1(\mathbb{S}^{n-1})$, $f\in L^1(\mathbb{R}^n)$ and
$B_{j}^{\eta,t}$ be as in $(2.1)$. Then, for any $s\geq1$ and
$\delta\in(0,1)$,
$$\max\Big\{\Big\|\sum\limits_{j\in\mathbb{Z}}\Theta_{\Omega,K_j,a}B_{j-s}^{\eta,t}\Big\|_1,\Big\|\sum\limits_{j\in\mathbb{Z}}\Theta_{\Omega,K_j^{s,\delta},a}B_{j-s}^{\eta,t}\Big\|_1\Big\}
\lesssim\|\hat{a}\|_1\|\Omega\|_{L^1(\mathbb{S}^{n-1})}\|f\|_1,$$
$$\Big\|\sum\limits_{j\in\mathbb{Z}}(\Theta_{\Omega,K_j,a}B_{j-s}^{\eta,t}-\Theta_{\Omega,K_j^{s,\delta},a}B_{j-s}^{\eta,t})\Big\|_1\lesssim s^{-2}\|\hat{a}\|_1\|\Omega\|_{L^1(\mathbb{S}^{n-1})}\|f\|_1.$$
\end{lemma}
\begin{proof}
Fix $s\geq1$ and $\delta\in(0,1)$. Observe that
$$|T_{\Omega,K_j}B_{j-s}^{\eta,t}(x)|\leq\int_{\mathbb{R}^n}|\Omega(x-y)K_j(x,y)B_{j-s}^{\eta,t}(y)|dy\lesssim W_{j,\Omega}*|B_{j-s}^{\eta,t}|(x),$$
where $W_{j,\Omega}(x)=2^{-jn}|\Omega(x)|\chi_{[2^{j-2},2^{j+2}]}(|x|)$.
Since $\|W_{j,\Omega}\|_1\lesssim\|\Omega\|_{L^1(\mathbb{S}^{n-1})}$, Fubini's theorem and (2.2) yield,
$$\begin{array}{ll}
&\displaystyle\Big\|\sum\limits_{j\in\mathbb{Z}}\Theta_{\Omega,K_j,a}B_{j-s}^{\eta,t}\|_1
\lesssim\displaystyle\int_{\mathbb{R}^n}\iint_{[0,1]\times\mathbb{R}^n}|a^{x,t}(\eta)|\sum\limits_{j\in\mathbb{Z}}W_{j,\Omega}*|B_{j-s}^{\eta,t}|(x)d\eta dtdx\\
&\qquad\qquad\qquad\qquad\lesssim\displaystyle\iint_{[0,1]\times\mathbb{R}^n}|\hat{a}(\eta)|\sum\limits_{j\in\mathbb{Z}}\int_{\mathbb{R}^n}W_{j,\Omega}*|B_{j-s}^{\eta,t}|(x)dxd\eta dt\\
&\qquad\qquad\qquad\qquad\lesssim\displaystyle\|\Omega\|_{L^1(\mathbb{S}^{n-1})}\iint_{[0,1]\times\mathbb{R}^n}|\hat{a}(\eta)|\sum\limits_{j\in\mathbb{Z}}\|B_{j-s}^{\eta,t}\|_1d\eta dt\\
&\qquad\qquad\qquad\qquad\lesssim\displaystyle\|\hat{a}\|_1\|\Omega\|_{L^1(\mathbb{S}^{n-1})}\|f\|_1.
\end{array}$$
On the other hand, it was shown in the proof of Lemma 2.1 in \cite{DL} that
$$\|T_{\Omega,K_j}B_{j-s}^{\eta,t}-T_{\Omega,K_j^{s,\delta}}B_{j-s}^{\eta,t}\|_1\lesssim s^{-2}\|\Omega\|_{L^1(\mathbb{S}^{n-1})}\|B_{j-s}^{\eta,t}\|_1.$$
This together with (2.2) and Fubini's theorem implies that
$$\begin{array}{ll}
&\quad\displaystyle\Big\|\sum\limits_{j\in\mathbb{Z}}(\Theta_{\Omega,K_j,a}B_{j-s}^{\eta,t}-\Theta_{\Omega,K_j^{s,\delta},a}B_{j-s}^{\eta,t})\Big\|_1\\
&\lesssim\displaystyle\int_{\mathbb{R}^n}\iint_{[0,1]\times\mathbb{R}^n}|a^{x,t}(\eta)|
\sum\limits_{j\in\mathbb{Z}}|T_{\Omega,K_j}u^{\eta,t}(x)-T_{\Omega,K_j^{s,\delta}}u^{\eta,t}(x)|d\eta dtdx\\
&\lesssim\displaystyle\iint_{[0,1]\times\mathbb{R}^n}|\hat{a}(\eta)|\sum\limits_{j\in\mathbb{Z}}\|T_{\Omega,K_j}B_{j-s}^{\eta,t}-T_{\Omega,K_j^{s,\delta}}B_{j-s}^{\eta,t}\|_1d\eta dt\\
&\lesssim\displaystyle s^{-2}\|\Omega\|_{L^1(\mathbb{S}^{n-1})}\iint_{[0,1]\times\mathbb{R}^n}|\hat{a}(\eta)|\sum\limits_{j\in\mathbb{Z}}\|B_{j-s}^{\eta,t}\|_1d\eta dt\\
&\lesssim s^{-2}\|\hat{a}\|_1\|\Omega\|_{L^1(\mathbb{S}^{n-1})}\|f\|_1.
\end{array}$$
Hence,
$$\Big\|\sum\limits_{j\in\mathbb{Z}}\Theta_{\Omega,K_j^{s,\delta},a}B_{j-s}^{\eta,t}\Big\|_1\lesssim\|\hat{a}\|_1\|\Omega\|_{L^1(\mathbb{S}^{n-1})}\|f\|_1.$$
Then Lemma \ref{lem2.2} is proved.
\end{proof}

Let $\phi$ be a nonnegative radial ${C}_0^\infty(\mathbb{R}^n)$
function such that ${\rm supp}(\phi)\subset\{x\in\mathbb{R}^n:|x|\leq1/4\}$
and $\int_{\mathbb{R}^n}\phi(x)dx=1$. For $r\in\mathbb{R}$ we set
$\phi_r(x)=2^{-rn}\phi(2^{-r}x)$ and define the convolution operator
$\Phi_r$ by
$$\Phi_rf(x)=\phi_r*f(x).\eqno(2.6)$$
For an operator $U$, $(\eta,t)\in\mathbb{R}^n\times[0,1]$ and $u^{\eta,t}:
\mathbb{R}^n\rightarrow\mathbb{R}$ we define the following operator
associated to $U$ by
$$\Gamma_{\Omega,K,a}^{U}u^{\eta,t}(x)=\iint_{[0,1]\times\mathbb{R}^n}a^{x,t}(\eta)UT_{\Omega,K}u^{\eta,t}(x)d\eta dt.\eqno(2.7)$$

\begin{lemma}\label{lem2.3}
Let $s\geq1$, $\delta\in(0,1)$, $\gamma\in(0,1)$ and
$\Omega\in L^1(\mathbb{S}^{n-1})$. Let $\Phi_r$ be as in $(2.6)$ and
$\Gamma_{\Omega,K,a}^{\Phi_r}$ be as in $(2.7)$. Let $b_Q^{\eta,t}$
be as in Lemma \ref{lem2.1} with $L(Q)=2^{j-s}$. Then we have that for any
$f\in L^1(\mathbb{R}^n)$,
$$\Big\|\sum\limits_{j\in\mathbb{Z}}\sum\limits_{L(Q)=2^{j-s}}\Gamma_{\Omega,K_j^{s,\delta},a}^{\Phi_{j-s\gamma}}b_Q^{\eta,t}\Big\|_1
\lesssim (2^{-s(1-\gamma)}+s^{2\delta^{-1}}2^{-s\delta})\|\hat{a}\|_1
\|\Omega\|_{L^1(\mathbb{S}^{n-1})}\|f\|_1.\eqno(2.8)$$
\end{lemma}
\begin{proof}
Let $y_0$ be the center of $Q$ and set
$$A(x,y):=\int_{\mathbb{R}^n}\phi_{j-s\gamma}(x-y-t)\Omega(t)K_j^{s,\delta}(y+t,y)dt.$$
By some changes of variables, we obtain
$$\begin{array}{ll}
&\Phi_{j-s\gamma}T_{\Omega,K_j^{s,\delta}}b_Q^{\eta,t}(x)=\displaystyle\int_{\mathbb{R}^n}
\phi_{j-s\gamma}(x-z)\int_{\mathbb{R}^n}\Omega(z-y)K_j^{s,\delta}(z,y)b_Q^{\eta,t}(y)dydz\\
&\qquad\qquad\qquad\qquad=\displaystyle\int_{\mathbb{R}^n}\phi_{j-s\gamma}(x-z)
\int_{\mathbb{R}^n}\Omega(z-y)K_j^{s,\delta}(z,y)dzb_Q^{\eta,t}(y)dy\\
&\qquad\qquad\qquad\qquad=\displaystyle\int_{\mathbb{R}^n}A(x,y)b_Q^{\eta,t}(y)dy.
\end{array}$$
By the fact $\int_{Q}b_Q^{\eta,t}(y)dy=0$, it holds that
$$\Phi_{j-s\gamma}T_{\Omega,K_j^{s,\delta}}b_Q^{\eta,t}(x)=\int_{Q}(A(x,y)-A(x,y_0))b_Q^{\eta,t}(y)dy.\eqno(2.9)$$

Next we prove that
$$\sup\limits_{y\in Q}\|A(\cdot,y)-A(\cdot,y_0)\|_1\lesssim(2^{-s(1-\gamma)}+s^{2\delta^{-1}}2^{-s\delta})\|\Omega\|_{L^1(\mathbb{S}^{n-1})}.\eqno(2.10)$$
Fix $y\in Q$. By some changes of variables and Fubini's theorem,
$$\begin{array}{ll}
&\quad\|A(\cdot,y)-A(\cdot,y_0)\|_1\\
&\leq\displaystyle\int_{\mathbb{S}^{n-1}}\int_0^\infty\int_{\mathbb{R}^n}|\phi_{j-s\gamma}(x-y-r\theta)
-\phi_{j-s\gamma}(x-y_0-r\theta)|dx\\
&\quad\times|K_j^{s,\delta}(y+r\theta,y)|r^{n-1}dr|\Omega(\theta)|d\sigma(\theta)\\
&\quad+\displaystyle\int_{\mathbb{S}^{n-1}}\int_0^\infty\int_{\mathbb{R}^n}|\phi_{j-s\gamma}(x-y_0-r\theta)|
|K_j^{s,\delta}(y+r\theta,y)-K_j^{s,\delta}(y_0+r\theta,y_0)|dx\\
&\quad\times r^{n-1}dr|\Omega(\theta)|d\sigma(\theta).
\end{array}$$
Hence,
$$\begin{array}{ll}
&\displaystyle\int_{\mathbb{S}^{n-1}}\int_0^\infty\int_{\mathbb{R}^n}|\phi_{j-s\gamma}(x-y-r\theta)
-\phi_{j-s\gamma}(x-y_0-r\theta)|dx\\
&\times|K_j^{s,\delta}(y+r\theta,y)|r^{n-1}dr|\Omega(\theta)|d\sigma(\theta)\lesssim 2^{-s(1-\gamma)}\|\nabla\phi\|_1\|\Omega\|_{L^1(\mathbb{S}^{n-1})}.
\end{array}$$
Using an argument analogous to that in Lemma 4.5 of \cite{DL},
$$|K_j^{s,\delta}(y+r\theta,y)-K_j^{s,\delta}(y_0+r\theta,y_0)|\lesssim (s^{2\delta^{-1}}+1)2^{-s-jn}+2^{-s\delta-jn}\lesssim s^{2\delta^{-1}}2^{-s\delta-jn}$$
It follows that
$$\begin{array}{ll}
&\quad\displaystyle\int_{\mathbb{S}^{n-1}}\int_0^\infty\int_{\mathbb{R}^n}|\phi_{j-s\gamma}(x-y_0-r\theta)|
|K_j^{s,\delta}(y+r\theta,y)-K_j^{s,\delta}(y_0+r\theta,y_0)|dx\\
&\quad\times r^{n-1}dr|\Omega(\theta)|d\sigma(\theta)\\
&\leq\displaystyle\int_{\mathbb{S}^{n-1}}\int_{2^{j-2}}^{2^{j+2}}
\int_{\mathbb{R}^n}|\phi_{j-s\gamma}(x-y_0-r\theta)|dxs^{2\delta^{-1}}2^{-s\delta-jn}
r^{n-1}dr|\Omega(\theta)|d\sigma(\theta)\\
&\lesssim s^{2\delta^{-1}}2^{-s\delta}\|\Omega\|_{L^1(\mathbb{S}^{n-1})}.
\end{array}$$
Thus, (2.10) holds. Combining (2.10) with (2.9) yields that
$$\|\Phi_{j-s\gamma}T_{\Omega,K_j^{s,\delta}}b_Q^{\eta,t}\|_1\lesssim
(2^{-s(1-\gamma)}+s^{2\delta^{-1}}2^{-s\delta})\|\Omega\|_{L^1(\mathbb{S}^{n-1})}\|b_Q^{\eta,t}\|_1.$$
This together with Fubini's theorem and (2.2) implies that
$$\begin{array}{ll}
&\quad\displaystyle\Big\|\sum\limits_{j\in\mathbb{Z}}\sum\limits_{L(Q)=2^{j-s}}\Gamma_{\Omega,K_j^{s,\delta},a}^{\Phi_{j-s\gamma}}b_Q^{\eta,t}\Big\|_1\\
&\lesssim\displaystyle \iint_{[0,1]\times\mathbb{R}^n}|a^{x,t}(\eta)|\sum\limits_{j\in\mathbb{Z}}\sum\limits_{L(Q)=2^{j-s}}\|\Phi_{j-s\gamma}T_{\Omega,K_j^{s,\delta}}b_Q^{\eta,t}\|_1d\eta dt\\
&\lesssim\displaystyle(2^{-s(1-\gamma)}+s^{2\delta^{-1}}2^{-s\delta})\|\Omega\|_{L^1(\mathbb{S}^{n-1})}
\iint_{[0,1]\times\mathbb{R}^n}|\hat{a}(\eta)|\sum\limits_{j\in\mathbb{Z}}\sum\limits_{L(Q)=2^{j-s}}\|b_Q^{\eta,t}\|_1d\eta dt\\
&\lesssim (2^{-s(1-\gamma)}+s^{2\delta^{-1}}2^{-s\delta})\|\hat{a}\|_1
\|\Omega\|_{L^1(\mathbb{S}^{n-1})}\|f\|_1.
\end{array}$$
This proves (2.8).
\end{proof}

We now introduce some notation, which follows from \cite{Se}.
For $s\geq1$ let $\mathfrak{E}^s=\{e_v^s\}$ be a collection
of unit vectors with mutual distance $>2^{-s-10}n^{-1}$ such
that for each $\theta\in\mathbb{S}^{n-1}$ there is an $e_v^s$
with $|\theta-e_v^s|\leq2^{-s-1}$. It was pointed out in
\cite{Se} that there exist disjoint measurable sets
$E_v^s\subset\mathbb{S}^{n-1}$ with $e_v^s \in E_v^s$,
$\text{diam}(E_v^s)\leq2^{-s}$ and $\bigcup_{v}E_v^s=\mathbb{S}^{n-1}$.
Then $\text{card}(\mathfrak{E}^s)\approx 2^{s(n-1)}$.

Let $\Psi$ be a $C_0^\infty(\mathbb{R})$ function such that
$0\leq\Psi(t)\leq1$, $\supp(\Psi)\subset[-4,4]$ and $\Psi(t)=1$
for $t\in[-2,2]$. Define the following multiplier operators by
$$\widehat{P_{s,v}h}(\xi)=\Psi(2^{\frac{s}{2}}\langle|\xi|^{-1}\xi,e_v^s\rangle)\hat{h}(\xi),$$
$$\widehat{Q_{s,v}h}(\xi)=(1-\Psi(2^{\frac{s}{2}}\langle|\xi|^{-1}\xi,e_v^s\rangle))\hat{h}(\xi),$$
where $h$ is a Schwartz function.

\begin{lemma}\label{lem2.4}
Let $\rho>0$, $s\geq1$, $\delta\in(0,1)$, $\beta>0$ and
$\Omega_v\in L^1(\mathbb{S}^{n-1})$ and
${\rm supp}(\Omega_v)\subset E_v^s\cap\{\theta\in\mathbb{S}^{n-1}:
|\Omega(\theta)|\leq\rho\}$. Let $\Lambda$ be a collection of
dyadic cubes $Q$ with disjoint interiors. For $Q\in\Lambda$ let
$h_Q^{\eta,t}$ be an integrable function supported in $Q$ satisfying
$$\frac{1}{|Q|}\int_{\mathbb{R}^n}|h_Q^{\eta,t}(x)|dx\leq\alpha,$$
for some $\alpha>0$. Set $H_j^{\eta,t}=\sum_{Q\in\Lambda,\ L(Q)=2^{j}}
h_Q^{\eta,t}$ and assume
$$\sup\limits_{[\eta,t]\in\mathbb{R}^n\times[0,1]}\sum\limits_{j\in\mathbb{Z}}\|H_j^{\eta,t}\|_1\leq\beta$$
for some $\beta>0$. Then for any function $\tau$ defined on $\mathbb{Z}$,
$$\Big\|\sum\limits_{v}\sum\limits_{j\in\mathbb{Z}}\Gamma_{\Omega_v,K_j^{s,\delta},a}^{P_{v,s}(I-\Phi_{\tau(j)})}H_{j-s}^{\eta,t}\Big\|_2
\lesssim2^{-s/4}\rho\alpha^{1/2}\|\hat{a}\|_1\beta^{1/2}.\eqno(2.11)$$
\end{lemma}
\begin{proof}
Fix a function $\tau$ defined on $\mathbb{Z}$. In \cite{CLW,Se} it was pointed out  that for every $\xi\neq 0$, the point $\xi$ belongs to at most
$C(n)2^{s(n-\frac{3}{2})}$ of ${\rm supp}(\widehat{P}_{s,v})$. It
follows that
$$\sup\limits_{\xi\neq0}\sum\limits_{v}|\Psi(2^{s/2}\langle|\xi|^{-1}\xi,e_v^s\rangle)|^2\lesssim_n2^{s(n-\frac{3}{2})}.$$
By Fubini's theorem, one obtains
$$\begin{array}{ll}
&\quad\displaystyle\Big\|\sum\limits_{v}\sum\limits_{j\in\mathbb{Z}}\Gamma_{\Omega_v,K_j^{s,\delta},a}^{P_{v,s}(I-\Phi_{\tau(j)})}H_{j-s}^{\eta,t}\Big\|_2\\
&\leq\displaystyle\iint_{[0,1]\times\mathbb{R}^n}|a^{x,t}(\eta)|\Big\|\sum\limits_{v}P_{v,s}\sum\limits_{j\in\mathbb{Z}}(I-\Phi_{\tau(j)})T_{\Omega_v,K_j^{s,\delta}}H_{j-s}^{\eta,t}\Big\|_2d\eta dt.
\end{array}$$
Applying Plancherel's theorem and the Cauchy--Schwarz inequality,
one obtains
$$\begin{array}{ll}
&\quad\displaystyle\Big\|\sum\limits_{v}P_{v,s}\sum\limits_{j\in\mathbb{Z}}(I-\Phi_{\tau(j)})T_{\Omega_v,K_j^{s,\delta}}H_{j-s}^{\eta,t}\Big\|_2^2\\
&=\displaystyle\int_{\mathbb{R}^n}\Big(\sum\limits_v\Psi(2^{\frac{s}{2}}\langle|\xi|^{-1}\xi,e_v^s\rangle)
\mathcal{F}\Big(\sum\limits_{j\in\mathbb{Z}}(I-\Phi_{\tau(j)})T_{\Omega_v,K_j^{s,\delta}}H_{j-s}^{\eta,t}\Big)(\xi)\Big)^2d\xi\\
&\leq\displaystyle\int_{\mathbb{R}^n}\Big(\sum\limits_v|\Psi(2^{\frac{s}{2}}\langle|\xi|^{-1}\xi,e_v^s\rangle)|^2
\Big)\Big(\sum\limits_v\Big|\mathcal{F}
\Big(\sum\limits_{j\in\mathbb{Z}}(I-\Phi_{\tau(j)})T_{\Omega_v,K_j^{s,\delta}}H_{j-s}^{\eta,t}\Big)(\xi)\Big|^2\Big)d\xi\\
&\lesssim\displaystyle 2^{s(n-\frac{3}{2})}\sum\limits_v\Big\|\sum\limits_{j\in\mathbb{Z}}(I-\Phi_{\tau(j)})T_{\Omega_v,K_j^{s,\delta}}H_{j-s}^{\eta,t}\Big\|_2^2.
\end{array}$$
Note that $\text{card}(\mathfrak{E}^s)\approx 2^{s(n-1)}$. Hence,
for (2.11) it suffices to show that
$$\Big\|\sum\limits_{j\in\mathbb{Z}}(I-\Phi_{\tau(j)})T_{\Omega_v,K_j^{s,\delta}}H_{j-s}^{\eta,t}\Big\|_2^2
\lesssim2^{-2s(n-1)}\rho^2\alpha\sum\limits_{j}\|H_{j-s}^{\eta,t}\|_1.\eqno(2.12)$$
In fact, once (2.12) was proved, we have
$$\begin{array}{ll}
&\quad\displaystyle\Big\|\sum\limits_{v}\sum\limits_{j\in\mathbb{Z}}\Gamma_{\Omega_v,K_j^{s,\delta},a}^{P_{v,s}(I-\Phi_{\tau(j)})}H_{j-s}^{\eta,t}\Big\|_2\\
&\lesssim\displaystyle\iint_{[0,1]\times\mathbb{R}^n}|a^{x,t}(\eta)|\Big(2^{s(n-\frac{3}{2})}2^{s(n-1)}2^{-2s(n-1)}\rho^2\alpha\sum\limits_{j}\|H_{j-s}^{\eta,t}\|_1\Big)^{1/2}d\eta dt\\
&\lesssim\displaystyle2^{-s/4}\rho\alpha^{1/2}\|\hat{a}\|_1\beta^{1/2}.
\end{array}$$

Now we prove (2.12). It follows by (2.4) that
$$|T_{\Omega_v,K_j^{s,\delta}}H_{j-s}^{\eta,t}(x)|\lesssim K_{j,s,v}*|H_{j-s}^{\eta,t}|(x),\eqno(2.13)$$
where $K_{j,s,v}(x)=2^{-jn}\rho\chi_{[2^{j-2},2^{j+2}]}(|x|)\chi_{E_v^s}(\frac{x}{|x|})$.
It follows by (2.13) that
$$\begin{array}{ll}
&\quad\displaystyle\Big\|\sum\limits_{j\in\mathbb{Z}}T_{\Omega_v,K_j^{s,\delta}}H_{j-s}^{\eta,t}\Big\|_2^2\\
&\leq\displaystyle\sum\limits_{j\in\mathbb{Z}}\|T_{\Omega_v,K_j^{s,\delta}}H_{j-s}^{\eta,t}\|_2^2
+\displaystyle2\sum\limits_{j\in\mathbb{Z}}\sum\limits_{i=-\infty}^{j-1}
\int_{\mathbb{R}^n}|T_{\Omega_v,K_j^{s,\delta}}H_{j-s}^{\eta,t}(x)T_{\Omega_v,K_i^{s,\delta}}H_{i-s}^{\eta,t}(x)|dx\\
&\leq\displaystyle2\sum\limits_{j\in\mathbb{Z}}\sum\limits_{i=-\infty}^{j}
\int_{\mathbb{R}^n}|T_{\Omega_v,K_j^{s,\delta}}H_{j-s}^{\eta,t}(x)T_{\Omega_v,K_i^{s,\delta}}H_{i-s}^{\eta,t}(x)|dx\\
&\lesssim\displaystyle\sum\limits_{j\in\mathbb{Z}}\sum\limits_{i=-\infty}^{j}
\int_{\mathbb{R}^n}\widetilde{K_{j,s,v}}*K_{i,s,v}*|H_{i-s}^{\eta,t}|(x)|H_{j-s}^{\eta,t}|(x)dx.
\end{array}\eqno(2.14)$$
We also note that
$$|\Phi_{\tau(j)}T_{\Omega_v,K_j^{s,\delta}}H_{j-s}^{\eta,t}(x)|\lesssim
\phi_{\tau(j)}*K_{j,s,v}*|H_{j-s}^{\eta,t}|(x).$$
An argument similar to (2.14) gives that
$$\begin{array}{ll}
&\quad\displaystyle\Big\|\sum\limits_{j\in\mathbb{Z}}\Phi_{\tau(j)}T_{\Omega_v,K_j^{s,\delta}}H_{j-s}^{\eta,t}
\Big\|_2^2\\
&\lesssim\displaystyle\sum\limits_{j\in\mathbb{Z}}\sum\limits_{i=-\infty}^{j}\int_{\mathbb{R}^n}
\widetilde{K_{j,s,v}}*K_{i,s,v}*|H_{i-s}^{\eta,t}|(x)\phi_{\tau(j)}*\widetilde{\phi_{\tau(i)}}*|H_{j-s}^{\eta,t}|(x)dx.
\end{array}\eqno(2.15)$$

By the fact that $|E_v^s|\leq C2^{-s(n-1)}$, we see that
$\|K_{j,s,v}\|_1\lesssim\rho|E_v^s|\lesssim\rho2^{-s(n-1)}$.
This together with the trivial estimate
$\|K_{j,s,v}\|_\infty\lesssim 2^{-jn}\rho$ implies that
$$\|\widetilde{K_{j,s,v}}*K_{i,s,v}\|_\infty\leq\|\widetilde{K_{j,s,v}}\|_\infty\|K_{i,s,v}\|_1
\lesssim2^{-jn-s(n-1)}\rho^2.\eqno(2.16)$$
By the support of $K_{j,s,v}$ and the fact that ${\rm diam}(E_v^s)\leq 2^{-s}$,
we see that ${\rm supp}(K_{j,s,v})\subset R_j$, where $R_j$ is a rectangle
centered at the origin with the $n-1$ sides of length $2^{j-s+3}$ and one
long side of length $2^{j+4}$. Thus, we have ${\rm supp}(\widetilde{K_{j,
s,v}}*K_{j,s,v})\subset R_j+R_j$. Since the original boxes are carefully
chosen, $R_j+R_j$ is also a rectangular box centered at 0 with $n-1$ short
sides of length $2^{j-s+10}$ and one long side of length $2^{j+10}$, the
long side being parallel to $e_{v}^{s}$, which is denoted by $\mathcal{R}_{j,
v}^s$. When $i\leq j$, we note that $R_i+R_j\subset R_j+R_j$. Then
${\rm supp}(\widetilde{K_{i,s,v}}*K_{j,s,v})\subset\mathcal{R}_{j,v}^{s}$.
For fixed $v$ and $i,\,j$ with $i\leq j$, it follows by (2.16) that
$$\widetilde{K_{j,s,v}}*K_{i,s,v}*|H_{i-s}^{\eta,t}|(x)\lesssim\displaystyle2^{-jn-s(n-1)}\rho^2
\int_{x+\mathcal{R}_{j,v}^{s}}|H_{i-s}^{\eta,t}|(y)dy.$$
Combining this with (2.14) and (2.15) implies that
$$\begin{array}{ll}
&\quad\displaystyle\Big\|\sum\limits_{j\in\mathbb{Z}}(I-\Phi_{\tau(j)})T_{\Omega_v,K_j^{s,\delta}}H_{j-s}^{\eta,t}\Big\|_2^2\\
&\lesssim\displaystyle \Big\|\sum\limits_{j\in\mathbb{Z}}T_{\Omega_v,K_j^{s,\delta}}H_{j-s}^{\eta,t}
\Big\|_2^2+\Big\|\sum\limits_{j\in\mathbb{Z}}\Phi_{\tau(j)}T_{\Omega_v,K_j^{s,\delta}}H_{j-s}^{\eta,t}
\Big\|_2^2\\
&\lesssim\displaystyle\rho^22^{-s(n-1)}\sum\limits_{j\in\mathbb{Z}}2^{-jn}
\sup\limits_{x\in\mathbb{R}^n}\sum\limits_{i=-\infty}^{j}\int_{x+\mathcal{R}_{j,v}^{s}}|H_{i-s}^{\eta,t}|(y)dy\\
&\quad\times\displaystyle\int_{\mathbb{R}^n}(|H_{j-s}^{\eta,t}|(x)+\phi_{\tau(j)}*\widetilde{\phi_{\tau(i)}}*|H_{j-s}^{\eta,t}|(x))dx.
\end{array}\eqno(2.17)$$
Note that the cubes $Q$ are disjoint. Then for all $x\in\mathbb{R}^{n}$,
$$\begin{array}{ll}
&\quad\displaystyle\sum\limits_{i=-\infty}^j\int_{x+\mathcal{R}_{j,v}^{s}}|H_{i-s}^{\eta,t}|(y)dy
\leq\displaystyle\sum\limits_{i=-\infty}^j\sum\limits_{\substack{L(Q)=2^{i-s},\atop Q \cap(x+\mathcal{R}_{j,v}^{s})\neq\emptyset}}\int_{Q}|h_{Q}^{\eta,t}(x)|dx\\
&\qquad\qquad\qquad\qquad\qquad\leq\displaystyle\sum\limits_{i=-\infty}^j\alpha\sum_{\substack{L(Q)=2^{i-s},\atop
Q \cap(x+\mathcal{R}_{j,v}^{s}) \neq \emptyset}}|Q|\\
&\qquad\qquad\qquad\qquad\qquad\leq\alpha|x+2n\mathcal{R}_{j,\nu}^{s}|\lesssim2^{jn-s(n-1)}\alpha.
\end{array}\eqno(2.18)$$
It follows from (2.17) and (2.18) that
$$\Big\|\sum\limits_{j\in\mathbb{Z}}(I-\Phi_{\tau(j)})T_{\Omega_v,K_j^{s,\delta}}H_{j-s}^{\eta,t}\Big\|_2^2\lesssim \rho^2\alpha2^{-2s(n-1)}
\sum\limits_{j\in\mathbb{Z}}\|H_{j-s}^{\eta,t}\|_1.$$
This gives (2.12) and completes the proof of Lemma \ref{lem2.4}.
\end{proof}

\begin{lemma}\label{lem2.5}
Let $\gamma\in(0,1)$, $\delta\in(0,1)$, $N_1\in\mathbb{Z}^{+}$,
$N>\frac{n}{2}$, $\epsilon_0\in(0,1)$, $\rho>0$, $s\geq1$ and
$\Omega_v\in L^1(\mathbb{S}^{n-1})$ and ${\rm supp}(\Omega_v)
\subset E_v^s\cap\{\theta\in\mathbb{S}^{n-1}:
|\Omega_v(\theta)|\leq\rho\}$. Let $f,\,B_{j}^{\eta,t}$ be as
in $(2.2)$. Then we have
$$\Big\|\sum\limits_{v}\sum\limits_{j\in\mathbb{Z}}\Gamma_{\Omega_v,K_j^{s,\delta},a}^{Q_{v,s}(I-\Phi_{j-s\gamma})}B_{j-s}^{\eta,t}\Big\|_1
\lesssim\rho\|\hat{a}\|_12^{-\frac{\upsilon s}{2}}\|f\|_1,\eqno(2.19)$$
where $\upsilon=\min\{(\epsilon_0-\frac{1}{2})N_1-N,\ \ 2(\gamma-\epsilon_0)\}$.
\end{lemma}
\begin{proof}
It follows by Fubini's theorem that
$$\begin{array}{ll}
&\quad\displaystyle\Big\|\sum\limits_{v}\sum\limits_{j\in\mathbb{Z}}\Gamma_{\Omega_v,K_j^{s,\delta},a}^{Q_{v,s}(I-\Phi_{j-s\gamma})}B_{j-s}^{\eta,t}\Big\|_1\\
&\lesssim\displaystyle\iint_{[0,1]\times\mathbb{R}^n}|a^{x,t}(\eta)|\Big\|\sum\limits_{v}\sum\limits_{j\in\mathbb{Z}}Q_{v,s}(I-\Phi_{j-s\gamma})T_{\Omega_v,K_j^{s,\delta}}B_{j-s}^{\eta,t}\Big\|_1d\eta dt.
\end{array}$$
Hence, for (2.19), it suffices to show that
$$\Big\|\sum\limits_{v}\sum\limits_{j\in\mathbb{Z}}Q_{v,s}(I-\Phi_{j-s\gamma})T_{\Omega_v,K_j^{s,\delta},a}B_{j-s}^{\eta,t}\Big\|_1\lesssim \rho2^{-\upsilon s}\|f\|_1.\eqno(2.20)$$

In order to prove (2.20), we introduce the multiplier operators $L_k$ and
$\tilde{L}_k$ by
$$\widehat{L_k f}(\xi)=\varpi_k(\xi)\hat{f}(\xi),\ \ \ \widehat{\tilde{L}_kf}(\xi)=\nu_k(\xi)\hat{f}(\xi),$$
where $\varpi_k(\xi)=\varphi(2^k\xi)-\varphi(2^{k+1}\xi)$ and
$\nu_k(\xi)=\psi(2^k\xi)$. Here $\varphi$ is a radial $C^\infty$
function such that $\varphi(\xi)=1$ for $|\xi|\leq1$, $\varphi(\xi)=0$ for
$|\xi|\geq2$ and $0\leq\varphi(\xi)\leq1$ for all $\xi\in\mathbb{R}^n$.
$\psi$ is a radial $C^\infty$ function such that $\psi(\xi)=1$
for $1/2\leq|\xi|\leq2$, $\supp(\psi)\subset[1/4,4]$ and $0\leq\psi\leq1$
for all $\xi\in\mathbb{R}^n$. Observe that $\varpi_k=\nu_k\varpi_k$ and
$L_k=\tilde{L}_kL_k$. Moreover, $\supp(\varpi_k)\subset\{\xi\in\mathbb{R}^n:
2^{-k-1}\leq|\xi|\leq 2^{-k+1}\}$ and $\sum_{k\in\mathbb{Z}}\varpi_k(\xi)=1$.
Write
$$Q_{v,s}(I-\Phi_{j-s\gamma})T_{\Omega_v,K_j^{s,\delta}}B_{j-s}^{\eta,t}
=\sum\limits_{k}(I-\Phi_{j-s\gamma})\tilde{L}_kQ_{v,s}L_kT_{\Omega_v,K_j^{s,\delta}}B_{j-s}^{\eta,t}.$$
It follows that
$$\begin{array}{ll}
&\quad\displaystyle\Big\|\sum\limits_{v}\sum\limits_{j\in\mathbb{Z}}Q_{v,s}(I-\Phi_{j-s\gamma})T_{\Omega_v,K_j^{s,\delta},a}B_{j-s}^{\eta,t}\Big\|_1\\
&\leq\displaystyle\sum\limits_{v}\sum\limits_{j\in\mathbb{Z}}\sum\limits_{k\geq j-[s\varepsilon_0]}\|(I-\Phi_{j-s\gamma})
\tilde{L}_kQ_{v,s}L_kT_{\Omega_v,K_j^{s,\delta}}B_{j-s}^{\eta,t}\|_1\\
&\quad+\displaystyle\sum\limits_{v}\sum\limits_{j\in\mathbb{Z}}\sum\limits_{k<j-[s\varepsilon_0]}\|(I-\Phi_{j-s\gamma})
\tilde{L}_kQ_{v,s}L_kT_{\Omega_v,K_j^{s,\delta}}B_{j-s}^{\eta,t}\|_1=:I_1+I_2.
\end{array}\eqno(2.21)$$

We first estimate $I_1$. Note that $(I-\Phi_{j-s\gamma})\tilde{L}_k$ is a
convolution operator with kernel $\int_{\mathbb{R}^n}(\mathcal{F}^{-1}
(\nu_k)(y)-\mathcal{F}^{-1}(\nu_k)(y-z))\phi_{j-s\gamma}(z)dz$.
By the mean value theorem, one gets
$$\|(I-\Phi_{j-s\gamma})\tilde{L}_k\|_{L^1\rightarrow L^1}
\leq\displaystyle\int_{\mathbb{R}^n}\int_{\mathbb{R}^n}
|\mathcal{F}^{-1}(\nu_k)(y)-\mathcal{F}^{-1}(\nu_k)(y-z)|\phi_{j-s\gamma}(z)dzdy
\lesssim 2^{j-s\gamma-k}.$$
We also note that $\|Q_{v,s}L_k\|_{L^1\rightarrow L^1}\lesssim1$ and
$$\|T_{\Omega_v,K_j^{s,\delta}}B_{j-s}^{\eta,t}\|_1\lesssim \rho2^{-s(n-1)}\|B_{j-s}^{\eta,t}\|_1.$$
Thus,
$$\begin{array}{ll}
&I_1\lesssim\displaystyle\sum\limits_{v}\sum\limits_{j\in\mathbb{Z}}\sum\limits_{k\geq j-[s\varepsilon_0]}\rho2^{-s(n-1)}2^{j-s\gamma-k}\|B_{j-s}^{\eta,t}\|_1\\
&\quad\lesssim\displaystyle\rho \sum\limits_{v}2^{-s(n-1)-s\gamma+[s\varepsilon_0]}\sum\limits_{j\in\mathbb{Z}}\|B_{j-s}^{\eta,t}\|_1\lesssim\rho 2^{-s(\gamma-\varepsilon_0)}\|f\|_1.
\end{array}\eqno(2.22)$$

Next we estimate $I_2$. We first prove that
$$\begin{array}{ll}
&\quad\|(I-\Phi_{j-s\gamma})\tilde{L}_kQ_{v,s}L_kT_{\Omega_v,K_j^{s,\delta}}B_{j-s}^{\eta,t}\|_1\\
&\lesssim s^{2\delta^{-1}N_1}2^{(\frac{s}{2}+k-j)N_1+sN}\|\Omega_v\|_{L^1(\mathbb{S}^{n-1})}\|B_{j-s}^{\eta,t}\|_1.
\end{array}\eqno(2.23)$$
Observe that
$$\|(I-\Phi_{j-s\gamma})\tilde{L}_k\|_{L^1\rightarrow L^1}\leq\|\mathcal{F}^{-1}(\nu_k)-\phi_{j-s\gamma}*\mathcal{F}^{-1}(\nu_k)\|_{1}
\lesssim 1$$
and
$$\|(I-\Phi_{j-s\gamma})\tilde{L}_kQ_{v,s}L_kT_{\Omega_v,K_j^{s,\delta}}B_{j-s}^{\eta,t}\|_1
\leq\|(I-\Phi_{j-s\gamma})\tilde{L}_k\|_{L^1\rightarrow L^1}\|Q_{v,s}L_kT_{\Omega,K_j^{s,\delta}}B_{j-s}^{\eta,t}\|_1.$$
Then we have
$$\|(I-\Phi_{j-s\gamma})\tilde{L}_kQ_{v,s}L_kT_{\Omega_v,K_j^{s,\delta}}B_{j-s}^{\eta,t}\|_1\lesssim \|Q_{v,s}L_kT_{\Omega,K_j^{s,\delta}}B_{j-s}^{\eta,t}\|_1.$$
Thus, for (2.23) it is enough to show that
$$\|Q_{v,s}L_kT_{\Omega,K_j^{s,\delta}}B_{j-s}^{\eta,t}\|_1\lesssim s^{2\delta^{-1}N_1}2^{(\frac{s}{2}+k-j)N_1+sN}\|\Omega_v\|_{L^1(\mathbb{S}^{n-1})}\|B_{j-s}^{\eta,t}\|_1.\eqno(2.24)$$
For convenience, we set $m_{s,k}(\xi)=(1-\Psi(2^{\frac{s}{2}}\langle|\xi|^{-1}\xi,
e_v^s\rangle))\beta_k(\xi)$. Write
$$\begin{array}{ll}
&Q_{v,s}L_kT_{\Omega_v,K_j^{s,\delta}}B_{j-s}^{\eta,t}(x)=\displaystyle\int_{\mathbb{R}^n}\int_{\mathbb{R}^n}m_{s,k}(\zeta)e^{i\langle x-z,\zeta\rangle}d\zeta\int_{\mathbb{R}^n}\Omega_v(z-y)K_j^{s,\delta}(z,y)B_{j-s}^{\eta,t}(y)dydz\\
&\qquad\qquad\qquad\qquad\quad=\displaystyle\int_{\mathbb{R}^n}D(x,y)B_{j-s}^{\eta,t}(y)dy,
\end{array}$$
where
$$D(x,y)=\int_{\mathbb{R}^n}m_{s,k}(\zeta)\int_{\mathbb{R}^n}e^{i\langle x-z,\zeta\rangle}\Omega_v(z-y)K_j^{s,\delta}(z,y)dzd\zeta.$$
By Fubini's theorem, one obtains
$$\|Q_{v,s}L_kT_{\Omega_v,K_j^{s,\delta}}B_{j-s}^{\eta,t}\|_1\leq\int_{\mathbb{R}^n}
\int_{\mathbb{R}^n}|D(x,y)|dx|B_{j-s}^{\eta,t}(y)|dy\leq\sup\limits_{y\in\mathbb{R}^n}\|D(\cdot,y)\|_1\|B_{j-s}^{\eta,t}\|_1.$$
Thus, for (2.24) it suffices to conclude that
$$\sup\limits_{y\in\mathbb{R}^n}\|D(\cdot,y)\|_1\lesssim
s^{2\delta^{-1}N_1}2^{(\frac{s}{2}+k-j)N_1+sN}\|\Omega_v\|_{L^1(\mathbb{S}^{n-1})}.\eqno(2.25)$$

Next, we prove (2.25). Fix $y\in\mathbb{R}^n$. By some changes of variables, we obtain 
$$D(x,y)=\int_{\mathbb{R}^n}\int_{\mathbb{S}^{n-1}}\Omega_v(\theta)\int_0^\infty e^{i\langle x-y-r\theta,\zeta\rangle}K_j^{s,\delta}(y+r\theta,y)r^{n-1}drm_{s,k}(\zeta)d\zeta d\sigma(\theta).$$
Integrating by parts $N_1$ times with respect to $r$, we obtain
$$\begin{array}{ll}
&\quad\displaystyle\int_0^\infty e^{i\langle x-y-r\theta,\zeta\rangle}K_j^{s,\delta}(y+r\theta,y)r^{n-1}dr\\
&=\displaystyle\int_0^\infty e^{i\langle x-y-r\theta,\zeta\rangle}\partial_r^{N_1}(K_j^{s,\delta}(y+r\theta,y)r^{n-1})
(i\langle\theta,\zeta\rangle)^{-N_1}dr.
\end{array}$$
It follows that
$$\begin{array}{ll}
&D(x,y)=\displaystyle\int_{\mathbb{S}^{n-1}}\Omega_v(\theta)\int_0^\infty \int_{\mathbb{R}^n}e^{i\langle x-y-r\theta,\zeta\rangle}
(m_{s,k}(\zeta)(i\langle\theta,\zeta\rangle)^{-N_1})d\zeta\\
&\qquad\qquad\times\displaystyle\partial_r^{N_1}(K_j^{s,\delta}(y+r\theta,y)r^{n-1})
drd\sigma(\theta).
\end{array}$$
Integrating by parts $N$ times with respect to $\zeta$, we obtain
$$\begin{array}{ll}
&\quad\displaystyle\int_{\mathbb{R}^n}e^{i\langle x-y-r\theta,\zeta\rangle}
(m_{s,k}(\zeta)(i\langle\theta,\zeta\rangle)^{-N_1})d\zeta\\
&=\displaystyle\int_{\mathbb{R}^n}e^{i\langle x-y-r\theta,\zeta\rangle}
\frac{(I-2^{-2k}\bigtriangleup_\zeta)^N(m_{s,k}(\zeta)(i\langle\theta,\zeta\rangle)^{-N_1})}
{(1+2^{-2k}|x-y-r\theta|^2)^N}d\zeta.
\end{array}$$
Hence, we have
$$\begin{array}{ll}
&D(x,y)=\displaystyle\int_{\mathbb{S}^{n-1}}\Omega_v(\theta)\int_0^\infty\partial_r^{N_1}
(K_j^{s,\delta}(y+r\theta,y)r^{n-1})\\
&\qquad\qquad\times\displaystyle\int_{\mathbb{R}^n}e^{i\langle x-y-r\theta,\zeta\rangle}
\frac{(I-2^{-2k}\bigtriangleup_\zeta)^N(m_{s,k}(\zeta)(i\langle\theta,\zeta\rangle)^{-N_1})}
{(1+2^{-2k}|x-y-r\theta|^2)^N}d\zeta drd\theta.
\end{array}\eqno(2.26)$$
As in (4.9) \cite{DL}, we have that for any multi-indices $\alpha$,
$$|\partial_x^\alpha K_j^{s,\delta}(x,y)|\lesssim 2^{-(j-l_\delta(s))|\alpha|-jn}.$$
This together with the fact that $2^{j-2}\leq r\leq 2^{j+2}$ implies that
$$\begin{array}{ll}
&\quad|\partial_r^{N_1}(K_j^{s,\delta}(y+r\theta,y)r^{n-1})|\\
&=\displaystyle\Big|\sum\limits_{i=N_1-n+1}^{N_1}C_{N_1}^i\partial_r^i(K_j^{s,\delta}(y+r\theta,y))
\partial_r^{N_1-i}(r^{n-1})\Big|\lesssim s^{2\delta^{-1}N_1}2^{-(1+N_1)j}.
\end{array}\eqno(2.27)$$
By an argument analogous to that used to obtain (4.12) in \cite{DL},
$$|(I-2^{-2k}\bigtriangleup_\zeta)^N(m_{s,k}(\zeta)(i\langle\theta,\zeta\rangle)^{-N_1})|
\lesssim_{N_1}2^{(\frac{s}{2}+k)N_1+sN}.\eqno(2.28)$$
Observe that
$$\int_{\mathbb{R}^n}\frac{2^{-kn}}{(1+2^{-2k}|x-y-r\theta|^2)^N}dx\lesssim_n1,\eqno(2.29)$$
since $N>\frac{n}{2}$. Note that ${\rm supp}(m_{s,k})\subset\{x\in\mathbb{R}^n:
2^{-k-1}\leq|x|\leq 2^{-k+1}\}$. It follows from (2.26)--(2.29) that
$$\begin{array}{ll}
&\|D(\cdot,y)\|_1\lesssim_{N_1}\displaystyle s^{2\delta^{-1}N_1}2^{-(1+N_1)j+(\frac{s}{2}+k)N_1+sN}\\
&\qquad\qquad\quad\times\displaystyle\int_{\mathbb{R}^n}
\int_{\mathbb{S}^{n-1}}|\Omega_v(\theta)|\int_{2^{j-2}}^{2^{j+2}}\int_{{\rm supp}(m_{s,k})}
\frac{1}{(1+2^{-2k}|x-y-r\theta|^2)^N}d\zeta drd\theta dx\\
&\qquad\qquad\lesssim s^{2\delta^{-1}N_1}2^{-N_1j+(\frac{s}{2}+k)N_1+sN}\|\Omega_v\|_{L^1(\mathbb{S}^{n-1})}.
\end{array}$$
This gives (2.25). We get from (2.23) that
$$\begin{array}{ll}
&I_2\lesssim\displaystyle\sum\limits_{v}\sum\limits_{j\in\mathbb{Z}}\sum\limits_{k<j-[s\varepsilon_0]}s^{2\delta^{-1}N_1}2^{-N_1j+(\frac{s}{2}+k)N_1+sN}\rho|E_v^s|\|B_{j-s}^{\eta,t}\|_1\\
&\quad\lesssim\displaystyle\rho s^{2\delta^{-1}N_1}2^{(\frac{s}{2}-[s\varepsilon_0])N_1+sN}\sum\limits_{v}|E_v^s|\sum\limits_{j\in\mathbb{Z}}\|B_{j-s}^{\eta,t}\|_1\\
&\quad\lesssim2^{-\frac{((\varepsilon_0-\frac{1}{2})N_1-N)s}{2}}\rho\|f\|_1.
\end{array}$$
This together with (2.21) and (2.22) yields (2.20).
\end{proof}

\bigskip

\section{Proof of Theorem \ref{thm1.1}}\label{S3}
We now prove Theorem \ref{thm1.1}.
Let $\Omega\in L\log L(\mathbb{S}^{n-1})$ and $f\in L^1(\mathbb{R}^n)$. It suffices to show that for a fixed $\lambda>0$,
$$\big|\{x\in\mathbb{R}^n:|T_{\Omega,K,a}f(x)|>\lambda\}\big|\lesssim\lambda^{-1}\|\hat{a}\|_1C_\Omega\|f\|_1.\eqno(3.1)$$
In what follows, we assume $\delta\in(0,1)$, $\sigma\in(0,\frac{1}{4})$,
$\epsilon_0\in(\frac{1}{2},\frac{3}{4})$, $\gamma\in(\epsilon_0+\sigma,1)$,
$N=[\frac{n}{2}]+1$, $N_1=[\frac{N+2\sigma}{\epsilon_0-1/2}]+1$, and
$\upsilon=\min\{(\epsilon_0-\frac{1}{2})N_1-N,2(\gamma-\epsilon_0)\}$.
Clearly, $\upsilon>2\sigma$. Given $\lambda>0$, the proof of (3.1) will be divided into four steps.

\medskip

{\bf Step 1. Reduction by Calder\'{o}n--Zygmund decomposition.} Let
$b^{\eta,t},\,g^{\eta,t},\,\tau$ be as in Lemma \ref{lem2.1} with
$\tau=\frac{\lambda}{\|\hat{a}\|_1\|\Omega\|_{L\log L(\mathbb{S}^{n-1})}}$
and $B_j^{\eta,t}$ be defined in (2.1).

\medskip
 
Let $\Theta_{\Omega,K,a}$ be defined in (2.5). In view of Lemma \ref{lem2.1}
and (1.5), we write
$$T_{\Omega,K,a}f(x)=\Theta_{\Omega,K,a}g^{\eta,t}(x)+\Theta_{\Omega,K,a}b^{\eta,t}(x).$$
It follows that
$$\begin{array}{ll}
&\quad\big|\{x\in\mathbb{R}^n:|T_{\Omega,K,a}f(x)|>\lambda\}\big|\\
&\leq\displaystyle\Big|\Big\{x\in\mathbb{R}^n:|\Theta_{\Omega,K,a}g^{\eta,t}(x)|>\frac{\lambda}{2}\Big\}\Big|
+\Big|\Big\{x\in\mathbb{R}^n:|\Theta_{\Omega,K,a}b^{\eta,t}(x)|>\frac{\lambda}{2}\Big\}\Big|.
\end{array}\eqno(3.2)$$
By Chebyshev's inequality, Minkowski's inequality, the $L^2$ boundedness
assumption for $T_{\Omega,K}$ and the properties of $g$, one gets
$$\begin{array}{ll}
&\quad\displaystyle\Big|\Big\{x\in\mathbb{R}^n:|\Theta_{\Omega,K,a}g^{\eta,t}(x)|>\frac{\lambda}{2}\Big\}\Big|\\
&\lesssim\displaystyle\lambda^{-2}\Big(\int_{\mathbb{R}^n}\Big|\iint_{[0,1]\times\mathbb{R}^n}a^{x,t}(\eta)T_{\Omega,K}g^{\eta,t}(x)d\eta dt\Big|^2dx\Big)\\
&\lesssim\displaystyle\lambda^{-2}\Big(\iint_{[0,1]\times\mathbb{R}^n}|\hat{a}(\eta)|\|T_{\Omega,K} g^{\eta,t}\|_2d\eta dt\Big)^2\\
&\lesssim\displaystyle\lambda^{-2}\Big(\iint_{[0,1]\times\mathbb{R}^n}|\hat{a}(\eta)|\|\Omega\|_{L\log L(\mathbb{S}^{n-1})}(\tau\|f\|_1)^{1/2}d\eta dt\Big)^2\\
&\lesssim\lambda^{-1}\|\hat{a}\|_1\|\Omega\|_{L\log L(\mathbb{S}^{n-1})}\|f\|_1.
\end{array}\eqno(3.3)$$
In view of (3.2) and (3.3), for (3.1) it suffices to show that
$$\Big|\Big\{x\in\mathbb{R}^n:|\Theta_{\Omega,K,a}b^{\eta,t}(x)|>\frac{\lambda}{2}\Big\}\Big|\lesssim\lambda^{-1}\|\hat{a}\|_1 C_\Omega\|f\|_1.\eqno(3.4)$$

{\bf Step 2. Reduction for (3.4).} Let $K_j$ be as in (2.3) 
and $\{Q\}_{Q\in\mathcal{B}}$ be as in Lemma \ref{lem2.1}. 
We set 
$$E=\bigcup\limits_{Q\in\mathcal{B}}200Q.$$
Write
$$\Theta_{\Omega,K,a}b^{\eta,t}(x)=\sum\limits_{s\in\mathbb{Z}}
\sum\limits_{j\in\mathbb{Z}}\Theta_{\Omega,K_j,a}B_{j-s}^{\eta,t}(x).\eqno(3.5)$$
Note that $T_{\Omega,K_j}B_{j-s}^{\eta,t}(x)=0$ if $s\leq 9$ and $x\in E^c$.
It follows that $\Theta_{\Omega,K_j,a}B_{j-s}^{\eta,t}(x)=0$ if $s\leq 9$
and $x\in E^c$. In view of (3.5), for (3.4) it suffices to prove that
$$\Big|\Big\{x\in E^c:\Big|\sum\limits_{s>9}\sum\limits_{j\in\mathbb{Z}}\Theta_{\Omega,K_j,a}B_{j-s}^{\eta,t}(x)\Big|>\frac{\lambda}{2}\Big\}\Big|
\lesssim\lambda^{-1}\|\hat{a}\|_1C_\Omega\|f\|_1.\eqno(3.6)$$
Let $K_j^{s,\delta}$ be as in (2.4). Therefore, for (3.6) it suffices
to show that
$$\begin{array}{ll}
&\quad\displaystyle\Big|\Big\{x\in E^c:\Big|\sum\limits_{s>9}\sum\limits_{j\in\mathbb{Z}}(\Theta_{\Omega,K_j,a}B_{j-s}^{\eta,t}(x)-\Theta_{\Omega,K_j^{s,\delta},a}B_{j-s}^{\eta,t}(x))\Big|>\frac{\lambda}{4}\Big\}\Big|\\
&\lesssim\lambda^{-1}\|\hat{a}\|_1C_\Omega\|f\|_1,
\end{array}\eqno(3.7)$$
$$\Big|\Big\{x\in E^c:\Big|\sum\limits_{s>9}\sum\limits_{j\in\mathbb{Z}}\Theta_{\Omega,K_j^{s,\delta},a}B_{j-s}^{\eta,t}(x)\Big|>\frac{\lambda}{4}\Big\}\Big|
\lesssim\lambda^{-1}\|\hat{a}\|_1C_\Omega\|f\|_1.\eqno(3.8)$$
By Chebyshev's inequality, Minkowski's inequality and Lemma \ref{lem2.2},
we have
$$\begin{array}{ll}
&\quad\displaystyle\Big|\Big\{x\in E^c:\Big|\sum\limits_{s>9}\sum\limits_{j\in\mathbb{Z}}(\Theta_{\Omega,K_j,a}B_{j-s}^{\eta,t}(x)-\Theta_{\Omega,K_j^{s,\delta},a}B_{j-s}^{\eta,t}(x))\Big|>\frac{\lambda}{4}\Big\}\Big|\\
&\lesssim\displaystyle\lambda^{-1}\sum\limits_{s>9}\Big\|\sum\limits_{j\in\mathbb{Z}}(\Theta_{\Omega,K_j,a}B_{j-s}^{\eta,t}-\Theta_{\Omega,K_j^{s,\delta},a}B_{j-s}^{\eta,t})\Big\|_1\\
&\lesssim\displaystyle \lambda^{-1}\sum\limits_{s>9}s^{-2}\|\hat{a}\|_1\|\Omega\|_{L^1(\mathbb{S}^{n-1})}\|f\|_1\lesssim \lambda^{-1}\|\hat{a}\|_1\|\Omega\|_{L^1(\mathbb{S}^{n-1})}\|f\|_1.
\end{array}$$
This proves (3.7).

{\bf Step 3. Reduction for (3.8).} We now prove (3.8). Set
$$D_{s}:=\{\theta\in\mathbb{S}^{n-1}:|\Omega(\theta)|\geq 2^{s\sigma}\|\Omega\|_{L^1(\mathbb{S}^{n-1})}\}.$$
We can write
$$\begin{array}{ll}
&\quad\displaystyle\Big|\Big\{x\in E^c:\Big|\sum\limits_{s>9}\sum\limits_{j\in\mathbb{Z}}\Theta_{\Omega,K_j^{s,\delta},a}B_{j-s}^{\eta,t}(x)\Big|>\frac{\lambda}{4}\Big\}\Big|\\
&\leq\displaystyle\Big|\Big\{x\in E^c:\Big|\sum\limits_{s>9}\sum\limits_{j\in\mathbb{Z}}\Theta_{\Omega\chi_{D_{s}},K_j^{s,\delta},a}B_{j-s}^{\eta,t}(x)\Big|>\frac{\lambda}{8}\Big\}\Big|\\
&\quad+\displaystyle\Big|\Big\{x\in E^c:\Big|\sum\limits_{s>9}\sum\limits_{j\in\mathbb{Z}}\Theta_{\Omega\chi_{D_{s}^c},K_j^{s,\delta},a}B_{j-s}^{\eta,t}(x)\Big|>\frac{\lambda}{8}\Big\}\Big|.
\end{array}$$
By Chebyshev's inequality, Minkowski's inequality and Lemma \ref{lem2.2},
we have
$$\begin{array}{ll}
&\quad\displaystyle\Big|\Big\{x\in E^c:\Big|\sum\limits_{s>9}\sum\limits_{j\in\mathbb{Z}}\Theta_{\Omega\chi_{D_{s}},K_j^{s,\delta},a}B_{j-s}^{\eta,t}(x)\Big|>\frac{\lambda}{8}\Big\}\Big|\\
&\lesssim\displaystyle\lambda^{-1}\sum\limits_{s>9}\sum\limits_{j\in\mathbb{Z}}\|\Theta_{\Omega\chi_{D_{s}},K_j^{s,\delta},a}B_{j-s}^{\eta,t}\|_1\\
&\lesssim\displaystyle\lambda^{-1}\|\hat{a}\|_1\sum\limits_{s>9}\|\Omega\chi_{D_{s}}\|_{L^1(\mathbb{S}^{n-1})}\|f\|_1.
\end{array}$$
Note that
$$\sum\limits_{s>9}\|\Omega\chi_{D_{s}}\|_{L^1(\mathbb{S}^{n-1})}\leq\displaystyle\int_{\mathbb{S}^{n-1}}|\Omega(\theta)|{\rm card}\Big\{s\in\mathbb{N}:2^{s\sigma}\leq \frac{|\Omega(\theta)|}{\|\Omega\|_{L^1(\mathbb{S}^{n-1})}}\Big\}d\sigma(\theta)\lesssim C_\Omega.$$
Consequently,
$$\Big|\Big\{x\in E^c:\Big|\sum\limits_{s>9}\sum\limits_{j\in\mathbb{Z}}\Theta_{\Omega\chi_{D_{s}},K_j^{s,\delta},a}B_{j-s}^{\eta,t}(x)\Big|>\frac{\lambda}{8}\Big\}\Big|
\lesssim\lambda^{-1}\|\hat{a}\|_1C_\Omega\|f\|_1.$$
Therefore, for (3.8) it is enough to prove that
$$\Big|\Big\{x\in E^c:\Big|\sum\limits_{s>9}\sum\limits_{j\in\mathbb{Z}}\Theta_{\Omega\chi_{D_{s}^c},K_j^{s,\delta},a}B_{j-s}^{\eta,t}(x)\Big|>\frac{\lambda}{8}\Big\}\Big|
\lesssim\lambda^{-1}\|\hat{a}\|_1C_\Omega\|f\|_1.\eqno(3.9)$$

{\bf Step 4. Proof of (3.9).} Let $\Phi_r$ be as in $(2.6)$ and $\Gamma_{\Omega,
K,a}^{\Phi_r}$, $\Gamma_{\Omega,K,a}^{I-\Phi_r}$ be as in (2.7) and (2.8),
respectively. Let $\gamma\in(0,1)$. Write
$$\Theta_{\Omega\chi_{D_{s}^c},K_j^{s,\delta},a}B_{j-s}^{\eta,t}
=\Gamma_{\Omega\chi_{D_{s}^c},K_j^{s,\delta},a}^{\Phi_{j-s\gamma}}B_{j-s}^{\eta,t}+\Gamma_{\Omega\chi_{D_{s}^c},K_j^{s,\delta},a}^{I-\Phi_{j-s\gamma}}B_{j-s}^{\eta,t}.$$
Then we obtain
$$\begin{array}{ll}
&\quad\displaystyle\Big|\Big\{x\in E^c:\Big|\sum\limits_{s>9}\sum\limits_{j\in\mathbb{Z}}\Theta_{\Omega\chi_{D_{s}^c},K_j^{s,\delta},a}B_{j-s}^{\eta,t}(x)\Big|>\frac{\lambda}{8}\Big\}\Big|\\
&\leq\displaystyle\Big|\Big\{x\in E^c:\Big|\sum\limits_{s>9}\sum\limits_{j\in\mathbb{Z}}\Gamma_{\Omega\chi_{D_{s}^c},K_j^{s,\delta},a}^{\Phi_{j-s\gamma}}B_{j-s}^{\eta,t}(x)\Big|>\frac{\lambda}{16}\Big\}\Big|\\
&\quad+\displaystyle\Big|\Big\{x\in E^c:\Big|\sum\limits_{s>9}\sum\limits_{j\in\mathbb{Z}}\Gamma_{\Omega\chi_{D_{s}^c},K_j^{s,\delta},a}^{I-\Phi_{j-s\gamma}}B_{j-s}^{\eta,t}(x)\Big|>\frac{\lambda}{16}\Big\}\Big|.
\end{array}$$
By Chebyshev's inequality, Minkowski's inequality and Lemma \ref{lem2.3},
we have
$$\begin{array}{ll}
&\quad\displaystyle\Big|\Big\{x\in E^c:\Big|\sum\limits_{s>9}\sum\limits_{j\in\mathbb{Z}}\Gamma_{\Omega\chi_{D_{s}^c},K_j^{s,\delta},a}^{\Phi_{j-s\gamma}}B_{j-s}^{\eta,t}(x)\Big|>\frac{\lambda}{16}\Big\}\Big|\\
&\lesssim\displaystyle\lambda^{-1}\sum\limits_{s>9}\Big\|\sum\limits_{j\in\mathbb{Z}}\Gamma_{\Omega\chi_{D_{s}^c},K_j^{s,\delta},a}^{\Phi_{j-s\gamma}}B_{j-s}^{\eta,t}\Big\|_1\\
&\lesssim\displaystyle\lambda^{-1}\sum\limits_{s>9}\Big\|\sum\limits_{j\in\mathbb{Z}}\sum\limits_{L(Q)=j-s}\Gamma_{\Omega\chi_{D_{s}^c},K_j^{s,\delta},a}^{\Phi_{j-s\gamma}}b_Q^{\eta,t}\Big\|_1\\
&\lesssim\displaystyle\lambda^{-1}\sum\limits_{s>9}(2^{-s(1-\gamma)}+s^{2\delta^{-1}}2^{-s\delta})\|\hat{a}\|_1\|\Omega\chi_{D_{s}^c}\|_{L^1(\mathbb{S}^{n-1})}\|f\|_1\\
&\lesssim\displaystyle\lambda^{-1}\sum\limits_{s>1}
(2^{-s(1-\gamma)}+s^{2\delta^{-1}}2^{-s\delta})\|\hat{a}\|_1\|\Omega\|_{L^1(\mathbb{S}^{n-1})}\|f\|_1\\
&\lesssim\lambda^{-1}\|\hat{a}\|_1\|\Omega\|_{L^1(\mathbb{S}^{n-1})}\|f\|_1.
\end{array}$$
Therefore, for (3.9) it is sufficient to show that
$$\Big|\Big\{x\in E^c:\Big|\sum\limits_{s>9}\sum\limits_{j\in\mathbb{Z}}\Gamma_{\Omega\chi_{D_{s}^c},K_j^{s,\delta},a}^{I-\Phi_{j-s\gamma}}B_{j-s}^{\eta,t}(x)\Big|>\frac{\lambda}{16}\Big\}\Big|
\lesssim \lambda^{-1}\|\hat{a}\|_1\|\Omega\|_{L^1(\mathbb{S}^{n-1})}\|f\|_1.\eqno(3.10)$$

Now we prove (3.10). Let $E_v^s$ and $P_{s,v}$, $Q_{s,v}$ be as in
Section \ref{S2}. Set
$$d_v(x)=\Omega\big(\frac{x}{|x|}\big)\chi_{E_v^s}\big(\frac{x}{|x|}\big)
\chi_{D_{s}^c}\big(\frac{x}{|x|}\big).$$
Then we can write
$$\Gamma_{\Omega\chi_{D_{s}^c},K_j^{s,\delta},a}^{I-\Phi_{j-s\gamma}}B_{j-s}^{\eta,t}
=\sum\limits_{v}\Gamma_{d_v,K_j^{s,\delta},a}^{I-\Phi_{j-s\gamma}}B_{j-s}^{\eta,t}=
\sum\limits_{v}\Gamma_{d_v,K_j^{s,\delta},a}^{P_{s,v}(I-\Phi_{j-s\gamma})}B_{j-s}^{\eta,t}
+\sum\limits_{v}\Gamma_{d_v,K_j^{s,\delta},a}^{Q_{s,v}(I-\Phi_{j-s\gamma})}B_{j-s}^{\eta,t}.$$
It follows that
$$\begin{array}{ll}
&\quad\displaystyle\Big|\Big\{x\in E^c:\Big|\sum\limits_{s>9}\sum\limits_{j\in\mathbb{Z}}\Gamma_{\Omega\chi_{D_{s}^c},K_j^{s,\delta},a}^{I-\Phi_{j-s\gamma}}B_{j-s}^{\eta,t}(x)\Big|>\frac{\lambda}{16}\Big\}\Big|\\
&\leq\displaystyle\Big|\Big\{x\in E^c:\Big|\sum\limits_{s>9}\sum\limits_{v}\sum\limits_{j\in\mathbb{Z}}\Gamma_{d_v,K_j^{s,\delta},a}^{P_{s,v}(I-\Phi_{j-s\gamma})}B_{j-s}^{\eta,t}(x)\Big|>\frac{\lambda}{32}\Big\}\Big|\\
&\quad+\displaystyle\Big|\Big\{x\in E^c:\Big|\sum\limits_{s>9}
\sum\limits_{v}\sum\limits_{j\in\mathbb{Z}}\Gamma_{d_v,K_j^{s,\delta},a}^{Q_{s,v}(I-\Phi_{j-s\gamma})}B_{j-s}^{\eta,t}(x)\Big|>\frac{\lambda}{32}\Big\}\Big|.
\end{array}$$
By Chebyshev's inequality, Minkowski's inequality and Lemma \ref{lem2.4}, we have
$$\begin{array}{ll}
&\quad\displaystyle\Big|\Big\{x\in E^c:\Big|\sum\limits_{s>9}\sum\limits_{v}\sum\limits_{j\in\mathbb{Z}}\Gamma_{d_v,K_j^{s,\delta},a}^{P_{s,v}(I-\Phi_{j-s\gamma})}B_{j-s}^{\eta,t}(x)\Big|>\frac{\lambda}{32}\Big\}\Big|\\
&\lesssim\displaystyle\lambda^{-2}\Big(\sum\limits_{s>9}
\Big\|\sum\limits_{v}\sum\limits_{j\in\mathbb{Z}}\Gamma_{d_v,K_j^{s,\delta},a}^{P_{s,v}(I-\Phi_{j-s\gamma})}B_{j-s}^{\eta,t}\Big\|_2\Big)^2\\
&\lesssim\displaystyle\lambda^{-2}\Big(\sum\limits_{s>9}2^{-s(1/4-\sigma)}\|\Omega\|_{L^1(\mathbb{S}^{n-1})}\|\hat{a}\|_1(\tau\|f\|_1)^{1/2}\Big)^2\\
&\lesssim\lambda^{-1}\|\hat{a}\|_1\|\Omega\|_{L\log L(\mathbb{S}^{n-1})}\|f\|_1.
\end{array}\eqno(3.11)$$
Note that $\frac{\upsilon}{2}-\sigma>0$. By Chebyshev's inequality,
Minkowski's inequality and Lemma \ref{lem2.5}, we have
$$\begin{array}{ll}
&\quad\displaystyle\Big|\Big\{x\in E^c:\Big|\sum\limits_{s>9}
\sum\limits_{v}\sum\limits_{j\in\mathbb{Z}}\Gamma_{d_v,K_j^{s,\delta},a}^{Q_{s,v}(I-\Phi_{j-s\gamma})}B_{j-s}^{\eta,t}(x)\Big|>\frac{\lambda}{32}\Big\}\Big|\\
&\lesssim\displaystyle\lambda^{-1}\sum\limits_{s>9}
\Big\|\sum\limits_{v}\sum\limits_{j\in\mathbb{Z}}\Gamma_{d_v,K_j^{s,\delta},a}^{Q_{s,v}(I-\Phi_{j-s\gamma})}B_{j-s}^{\eta,t}\Big\|_1\\
&\lesssim\displaystyle\lambda^{-1}\|\hat{a}\|_1\sum\limits_{s>9}2^{-(\frac{\upsilon}{2}-\sigma)s}\|\Omega\|_{L^1(\mathbb{S}^{n-1})}\|f\|_1\\
&\lesssim \lambda^{-1}\|\hat{a}\|_1\|\Omega\|_{L^1(\mathbb{S}^{n-1})}\|f\|_1.
\end{array}\eqno(3.12)$$
Combining (3.12) with (3.11) yields (3.10), which completes the proof of
Theorem \ref{thm1.1}.$\hfill\Box$

\bigskip

\quad\hspace{-20pt}{\bf Acknowledgements.}
Feng Liu was supported by the Natural Science Foundation of Shandong Province
(grant No. ZR2023MA022) and the National Natural Science Foundation of China
(grant No. 12326371). Minglei Shi was supported by the National Natural Science
Foundation of China (grant No. 12571108) and the start-up foundation of Zhejiang Shuren University (grant No. 2023R028). The authors appreciate the reviewers for their helpful comments and suggestions that have substantially improved the presentation of this paper.

\medskip

\quad\hspace{-20pt}{\bf Data availability.} No data was used to support this study.

\end{document}